\documentclass[reqno]{amsart}
\usepackage{amssymb,amsmath,hyperref}
\usepackage{amsrefs}
\usepackage[foot]{amsaddr}
\usepackage{bbold,stackrel}
 
\newtheorem{thm}{Theorem}

\newtheorem{cor}[thm]{Corollary}
\newtheorem{defi}[thm]{Definition}
\newtheorem{rem}[thm]{Remark}
\newtheorem{nota}[thm]{Notation}

\newtheorem{princ}[thm]{Principle}

\newtheorem{ack}[thm]{Acknowledgement}

\newtheorem*{tempo*}{Template}
\newtheorem{theorem}[thm]{Theorem}

\newtheorem{lemma}[thm]{Lemma}
\newtheorem{definition}[thm]{Definition}

\newtheorem{proposition}[thm]{Proposition}

\newtheorem{problem}{Problem}

\newcommand\be{\begin{equation}}
\newcommand\ee{\end{equation}} 

\usepackage{amsmath,amsfonts} 

\usepackage[applemac]{inputenc}

\def\bdefi{\begin{defi}\rm}
\def\edefi{\end{defi}}
\def\bnota{\begin{nota}\rm}
\def\enota{\end{nota}}
\def\FIVE{\Pi_{1}^{1}\text{-\textup{\textsf{CA}}}_{0}}

\def\SIX{\Pi_{2}^{1}\text{-\textsf{\textup{CA}}}_{0}}

\def\SIXk{\Pi_{k}^{1}\text{-\textsf{\textup{CA}}}_{0}}
\def\SIXK{\Pi_{k}^{1}\text{-\textsf{\textup{CA}}}_{0}^{\omega}}

\def\ATR{\textup{\textsf{ATR}}}

\def\PITK{\Pi_{k}^{1}\textup{-\textsf{TR}}_{0}}
\def\TR{\textup{\textsf{TR}}}

\def\Z{\textup{\textsf{Z}}}

\def\LUS{\textup{\textsf{LUS}}}
\def\EGO{\textup{\textsf{EGO}}}

\def\ZFC{\textup{\textsf{ZFC}}}

\def\seq{\textup{\textsf{seq}}}
\def\type{\textup{\textsf{type}}}
\def\RM{\textup{\textsf{RM}}}
\def\closed{\textup{\textsf{closed}}}
\def\open{\textup{\textsf{open}}}

 \def\r{\mathbb{r}}

\def\c{\textup{\textsf{c}}}
\def\RCA{\textup{\textsf{RCA}}}
\def\MCT{\textup{\textsf{MCT}}}
\def\net{\textup{\textsf{net}}}

\def\MCT{\textup{\textsf{MCT}}}

\def\({\textup{(}}
\def\){\textup{)}}
\def\WO{\textup{\textsf{WO}}}
\def\RCAo{\textup{\textsf{RCA}}_{0}^{\omega}}
\def\ACAo{\textup{\textsf{ACA}}_{0}^{\omega}}

\def\WKL{\textup{\textsf{WKL}}}

\def\WWKL{\textup{\textsf{WWKL}}}
\def\bye{\end{document}}
\def\N{{\mathbb  N}}
\def\Q{{\mathbb  Q}}
\def\R{{\mathbb  R}}
\def\L{\textsf{\textup{L}}}

\def\di{\rightarrow}
\def\asa{\leftrightarrow}
\def\ACA{\textup{\textsf{ACA}}}

\def\QFAC{\textup{\textsf{QF-AC}}}

\def\LLP{\textup{\textsf{LLP}}}

\def\cocode{\textup{\textsf{cocode}}}
\def\M{\mathcal{M}}
\def\HBU{\textup{\textsf{HBU}}}
\def\SUS{\textup{\textsf{SUS}}}
\def\HEC{\textup{\textsf{HEC}}}
\def\SS{\textup{\bf{S}}}

\def\IND{\textup{\textsf{IND}}}

\def\blambda{\pmb{\lambda}}
\def\sg{\textup{\textsf{sg}}}

\def\LIN{\textup{\textsf{LIN}}}
\def\WHBU{\textup{\textsf{WHBU}}}

\def\MCT{\textup{\textsf{MCT}}}

\def\eps{\varepsilon}

\def\ECF{\textup{\textsf{ECF}}}

\def\SFF{\textup{\textsf{SFF}}}

\newcommand{\m}{{\bf m}}

\usepackage{graphicx}
\usepackage{tikz}
\usetikzlibrary{matrix, shapes.misc}

\numberwithin{equation}{section}
\numberwithin{thm}{section}

\usepackage{comment}

\begin{document}
\title[On the Vitali covering theorem]{On the logical and computational properties of the Vitali covering theorem}
\author{Dag Normann}
\address{Department of Mathematics, The University of Oslo, Norway}
\email{dnormann@math.uio.no}
\author{Sam Sanders}
\address{Department of Philosophy II, RUB Bochum, Germany}
\email{sasander@me.com}

\begin{abstract}
We study a version of the Vitali covering theorem, which we call $\WHBU$ and which is a direct weakening of the Heine-Borel theorem for uncountable coverings, called $\HBU$. 
We show that $\WHBU$ is central to measure theory by deriving it from various central approximation results related to \emph{Littlewood's three principles}.
A natural question is then \emph{how hard} it is to prove $\WHBU$ (in the sense of Kohlenbach's \emph{higher-order Reverse Mathematics}), and \emph{how hard} it is to compute the objects claimed to exist by $\WHBU$ (in the sense of Kleene's schemes S1-S9).    
The answer to both questions is `extremely hard', as follows: on one hand, in terms of the usual scale of (conventional) comprehension axioms, $\WHBU$ is only provable using Kleene's $\exists^{3}$, which implies full second-order arithmetic.    
On the other hand, realisers (aka witnessing functionals) for $\WHBU$, so-called $\Lambda$-functionals, are computable from Kleene's $\exists^{3}$, but not from weaker comprehension functionals.   
Despite this hardness, we show that $\WHBU$, and certain $\Lambda$-functionals, behave much better than $\HBU$ and the associated class of realisers, called $\Theta$-functionals.  
In particular, we identify a specific $\Lambda$-functional called $\Lambda_{\SS}$ which adds no computational power to the \emph{Suslin functional}, in contrast to $\Theta$-functionals.  
Finally, we introduce a hierarchy involving $\Theta$-functionals and $\HBU$.  % for second-order arithmetic.   
\end{abstract}
%
%\setcounter{page}{0}
%\tableofcontents
%\thispagestyle{empty}
%\newpage

\maketitle
\thispagestyle{empty}

\section{Introduction}\label{intro}
The most apt counterpart in mathematical logic of the commonplace \emph{one cannot fit a square peg into a round hole} is perhaps the following: a Turing machine cannot directly access third-order objects, like e.g.\ measurable functions.  
Thus, the development of measure theory in any framework based on Turing computability must proceed via second-order \emph{stand-ins} for higher-order objects.  
In particular, the following frameworks, (somehow) based on Turing computability, proceed by studying the computational properties of certain \emph{countable representations} of measurable objects: Reverse Mathematics (\cite{simpson2}*{X.1}), constructive analysis\footnote{Note that Bishop's constructive analysis is not based on Turing computability \emph{directly}, but one of its `intended models' is however (constructive) recursive mathematics (see \cite{brich}).  One aim of Feferman's predicative analysis is to capture constructive reasoning in the sense of Bishop.\label{dug}} (\cite{beeson1}*{I.13} for an overview), predicative analysis$^{\ref{dug}}$ (\cite{littlefef}), and computable analysis (\cite{bewierook}).   

\smallskip

The \emph{existence} of the aforementioned countable representations is guaranteed by various well-known approximation results.  Perhaps the most basic and best-known among these results go by the name of \emph{Littlewood's three principles}.  The latter are found throughout the literature, including Tao's \emph{introduction to measure theory} (see \cites{kesteisdenbeste,royden1, steengoed, brezen, yuppie, taomes}), and were originally formulated by Littlewood as:
\begin{quote}
There are three principles, roughly expressible in the following terms: Every (measurable) set is nearly a finite sum of intervals; every function (of class $L^{p}$) is nearly continuous; every convergent sequence of functions is nearly uniformly convergent.
 (\cite{kleinbos}*{p.\ 26})
\end{quote}
The second and third principle are heuristic descriptions of the \emph{Lusin and Egorov theorems}. 
In light of their fundamental role for measure theory, it is then a natural question \emph{how hard} it is to prove these theorems, in the sense of Kohlenbach's \emph{higher-order Reverse Mathematics} (RM hereafter; see Section \ref{prelim1}), and \emph{how hard} it is to compute the countable approximations therein, in the sense of Kleene's schemes S1-S9 (see Section \ref{prelim2}).  The aim of this paper is to answer these connected questions.  
As it turns out, the answer to both questions is `extremely hard', as follows.  

\smallskip

In Section \ref{RMM}, we show that the aforementioned approximation theorems (and related results) imply $\WHBU$ as in Principle \ref{ohyee}, which is a version of the Vitali covering theorem that is a direct weakening of $\HBU$; the latter is
the Heine-Borel theorem for uncountable coverings as in Principle~\ref{drifty}.
%in as $\WHBU$ in Principle \ref{ohyee} and discussed in Section \ref{introk}.
In terms of standard comprehension axioms, $\WHBU$ is only provable using Kleene's $\exists^{3}$, which implies full second-order arithmetic (see Section \ref{moresubsec}).  
%However, $\WHBU$ is strictly weaker than $\HBU$.
Our approach to measure theory is akin to that of second-order RM (see Remark \ref{dreft}), but we shall also study the framework from \cite{elkhuisje}, namely in Section~\ref{onderkruipers}. 
We show in Section \ref{GL} that our results pertaining to $\WHBU$ are \emph{robust}, in that they do not depend on the framework at hand.  
%We discuss these results in more detail in Section \ref{intro1}.

\smallskip

In Section \ref{CTM}, we will study the computational properties of realisers of $\WHBU$, called\footnote{Like for Heine-Borel compactness, there is no \emph{unique} realiser for $\WHBU$ as in Principle \ref{ohyee}, as we can always add dummy elements to the sub-cover at hand.} 
 \emph{weak fan functionals} or \emph{$\Lambda$-functionals} (see \cite{dagsam, dagsamII}).  Any $\Lambda$-functional is computable from $\exists^{3}$, but not from weaker comprehension functionals like $\SS_{k}^{2}$ that decide $\Pi_{k}^{1}$-formulas (see Section \ref{moresubsec}).  
Despite this observed hardness, we show that $\WHBU$ and $\Lambda$-functionals behave much better than (Heine-Borel) compactness and the associated class of realisers, called \emph{special fan functionals} or \emph{$\Theta$}-functionals.  
In particular, we identify a specific $\Lambda$-functional, called $\Lambda_{\SS}$, which adds no computational power to the \emph{Suslin functional}, in contrast\footnote{It is shown in \cite{dagsamIII, dagcie18} that $\Theta$-functionals yields realisers for $\ATR_{0}$ when combined with the Turing jump functional $\exists^{2}$ from Section~\ref{prelim2}; $\Theta$-functionals also yield Gandy's \emph{Superjump} $\mathbb{S}$, and even fixed points of non-monotone inductive definitions, when combined with the Suslin functional. } to $\Theta$-functionals.  
As an application, we show that higher-order $\FIVE$ plus $\WHBU$ cannot prove $\HBU$.
We also show that $\Theta$-functionals and (Heine-Borel) compactness yield new hierarchies akin to second-order arithmetic in Section \ref{kew}.

\smallskip

In Section \ref{konkelfoes}, we formulate the conclusion to this paper as follows: we discuss a conjecture and a template related to our results in Section \ref{tempy}, while an interesting `dichotomy' phenomenon is observed in Section \ref{dich}.
In Section \ref{muse}, we discuss some foundational musings related to the coding practise of Reverse Mathematics.  

\smallskip

Finally, some new insights regarding `normal' and `non-normal' mathematics have recently come to the fore in e.g.\ \cites{dagsamX, samph}, 
providing a more `grand scheme of things' view of the results in this paper, as discussed in Remark \ref{dorf}.  In a nutshell, the above results 
should be viewed as motivation for the development and study of a new scale not based on comprehension or discontinuous functionals.  

\section{Preliminaries}\label{prelim}
We introduce \emph{Reverse Mathematics} in Section \ref{prelim1}, as well as its generalisation to \emph{higher-order arithmetic}.  
In particular, since we shall study measure theory, we discuss the representation of sets in Section \ref{prelim1}.
As our main results are proved using techniques from \emph{computability theory}, we discuss the latter in Section~\ref{prelim2}.
\subsection{Reverse Mathematics}\label{prelim1}
\subsubsection{Introduction}
Reverse Mathematics (RM hereafter) is a program in the foundations of mathematics initiated around 1975 by Friedman (\cites{fried,fried2}) and developed extensively by Simpson (\cite{simpson2}).  
The aim of RM is to identify the minimal axioms needed to prove theorems of ordinary, i.e.\ non-set theoretical, mathematics. 
We refer to \cite{stillebron} for a basic introduction to RM and to \cite{simpson2, simpson1} for an overview of RM.  We expect basic familiarity with RM, but do sketch some aspects of Kohlenbach's \emph{higher-order} RM (\cite{kohlenbach2}) essential to this paper, including the `base theory' $\RCAo$ in Section \ref{allyourbases2}.  Since we shall study measure theory, we need to represent sets in $\RCAo$, as discussed in Definition \ref{keepintireal}.\eqref{strijker} and (in more detail) Section \ref{kuiken}.    

\smallskip

Now, `classical' RM is based on $\L_{2}$, the language of \emph{second-order arithmetic} $\Z_{2}$.  By contrast, higher-order RM makes use of the richer language of \emph{higher-order arithmetic}.  
Indeed, while $\L_{2}$ is restricted to natural numbers and sets of natural numbers, higher-order arithmetic can accommodate sets of sets of natural numbers, sets of sets of sets of natural numbers, et cetera.  
To formalise this idea, we introduce the collection of \emph{all finite types} $\mathbf{T}$, defined by the two clauses:
\begin{center}
(i) $0\in \mathbf{T}$   and   (ii)  If $\sigma, \tau\in \mathbf{T}$ then $( \sigma \di \tau) \in \mathbf{T}$,
\end{center}
where $0$ is the type of natural numbers, and $\sigma\di \tau$ is the type of mappings from objects of type $\sigma$ to objects of type $\tau$.
In this way, $1\equiv 0\di 0$ is the type of functions from numbers to numbers, and where  $n+1\equiv n\di 0$.  Viewing sets as given by characteristic functions, we note that $\Z_{2}$ only includes objects of type $0$ and $1$.    

\smallskip

The language $\L_{\omega}$ includes variables $x^{\rho}, y^{\rho}, z^{\rho},\dots$ of any finite type $\rho\in \mathbf{T}$.  Types may be omitted when they can be inferred from context.  
The constants of $\L_{\omega}$ includes the type $0$ objects $0, 1$ and $ <_{0}, +_{0}, \times_{0},=_{0}$  which are intended to have their usual meaning as operations on $\N$.
Equality at higher types is defined in terms of `$=_{0}$' as follows: for any objects $x^{\tau}, y^{\tau}$, we have
\be\label{aparth}
[x=_{\tau}y] \equiv (\forall z_{1}^{\tau_{1}}\dots z_{k}^{\tau_{k}})[xz_{1}\dots z_{k}=_{0}yz_{1}\dots z_{k}],
\ee
if the type $\tau$ is composed as $\tau\equiv(\tau_{1}\di \dots\di \tau_{k}\di 0)$.  
Furthermore, $\L_{\omega}$ also includes the \emph{recursor constant} $\mathbf{R}_{\sigma}$ for any $\sigma\in \mathbf{T}$, which allows for iteration on type $\sigma$-objects as in the special case \eqref{special}.  
Formulas and terms are defined as usual.  

\subsubsection{The base theory of higher-order Reverse Mathematics}\label{allyourbases2}
We introduce the base theory $\RCAo$ of higher-order RM and discuss its connection to $\RCA_{0}$, the base theory of second-order RM. 
\bdefi\label{allyourbases}
The base theory $\RCAo$ consists of the following axioms.
\begin{enumerate}
\item  Basic axioms expressing that $0, 1, <_{0}, +_{0}, \times_{0}$ form an ordered semi-ring with equality $=_{0}$.
\item Basic axioms defining the well-known $\Pi$ and $\Sigma$ combinators (aka $K$ and $S$ in \cite{avi2}), which allow for the definition of \emph{$\lambda$-abstraction}. 
\item The defining axiom of the recursor constant $\mathbf{R}_{0}$: For $m^{0}$ and $f^{1}$: 
\be\label{special}
\mathbf{R}_{0}(f, m, 0):= m \textup{ and } \mathbf{R}_{0}(f, m, n+1):= f(n, \mathbf{R}_{0}(f, m, n)).
\ee
\item The \emph{axiom of extensionality}: for all $\rho, \tau\in \mathbf{T}$, we have:
\be\label{EXT}\tag{$\textsf{\textup{E}}_{\rho, \tau}$}  
(\forall  x^{\rho},y^{\rho}, \varphi^{\rho\di \tau}) \big[x=_{\rho} y \di \varphi(x)=_{\tau}\varphi(y)   \big].
\ee 
\item The induction axiom for quantifier-free\footnote{To be absolutely clear, variables (of any finite type) are allowed in quantifier-free formulas of the language $\L_{\omega}$: only quantifiers are banned.} formulas of $\L_{\omega}$.
\item $\QFAC^{1,0}$: The quantifier-free Axiom of Choice as in Definition \ref{QFAC}.
\end{enumerate}
\edefi
\bdefi\label{QFAC} The axiom $\QFAC$ consists of the following for all $\sigma, \tau \in \textbf{T}$:
\be\tag{$\QFAC^{\sigma,\tau}$}
(\forall x^{\sigma})(\exists y^{\tau})A(x, y)\di (\exists Y^{\sigma\di \tau})(\forall x^{\sigma})A(x, Y(x)),
\ee
for any quantifier-free formula $A$ in the language of $\L_{\omega}$.
\edefi
Recursion as in \eqref{special} is called \emph{primitive recursion}; the class of functionals obtained from $\mathbf{R}_{\rho}$ for all $\rho \in \mathbf{T}$ is called \emph{G\"odel's system $T$} of all (higher-order) primitive recursive functionals.  

\smallskip

Finally, as discussed in \cite{kohlenbach2}*{\S2}, $\RCAo$ and $\RCA_{0}$ prove the same sentences `up to language' as the latter is set-based and the former function-based.  
This is proved via the highly useful $\ECF$-interpretation, discussed next.  
\begin{rem}[The $\ECF$-interpretation]\label{ECF}\rm
The technical definition of $\ECF$ may be found in \cite{troelstra1}*{p.\ 138, \S2.6}.
Intuitively speaking, the $\ECF$-interpretation $[A]_{\ECF}$ of a formula $A\in \L_{\omega}$ is just $A$ with all variables 
of type two and higher replaced by countable representations of continuous functionals.  Such representations are also (equivalently) called `associates' or `codes' (see \cite{kohlenbach4}*{\S4}). 
The $\ECF$-interpretation connects $\RCAo$ and $\RCA_{0}$ (see \cite{kohlenbach2}*{Prop.\ 3.1}) in that if $\RCAo$ proves $A$, then $\RCA_{0}$ proves $[A]_{\ECF}$, again `up to language', as $\RCA_{0}$ is 
formulated using sets, and $[A]_{\ECF}$ is formulated using types, namely only using type zero and one objects.  
\end{rem}

\subsubsection{Basic definitions}
We list some basic definitions and notations needed below.

\smallskip

Firstly, we use the usual notations for natural, rational, and real numbers, and the associated functions, as introduced in \cite{kohlenbach2}*{p.\ 288-289}.  
\begin{defi}[Real numbers and related notions in $\RCAo$]\label{keepintireal}\rm~
\begin{enumerate}
\renewcommand{\theenumi}{\roman{enumi}}
\item Natural numbers correspond to type zero objects, and we use `$n^{0}$' and `$n\in \N$' interchangeably.  Rational numbers are defined as signed quotients of natural numbers, and `$q\in \Q$' and `$<_{\Q}$' have their usual meaning.    
\item Real numbers are represented by fast-converging Cauchy sequences $q_{(\cdot)}:\N\di \Q$, i.e.\  such that $(\forall n^{0}, i^{0})(|q_{n}-q_{n+i})|<_{\Q} \frac{1}{2^{n}})$.  
We use the `hat function' from \cite{kohlenbach2}*{p.\ 289} to guarantee that every $f^{1}$ defines a real number.  
\item We write `$x\in \R$' to express that $x^{1}:=(q^{1}_{(\cdot)})$ represents a real as in the previous item and write $[x](k):=q_{k}$ for the $k$-th approximation of $x$.    
\item Two reals $x, y$ represented by $q_{(\cdot)}$ and $r_{(\cdot)}$ are \emph{equal}, denoted $x=_{\R}y$, if $(\forall n^{0})(|q_{n}-r_{n}|\leq {2^{-n+1}})$. Inequality `$<_{\R}$' is defined similarly.  
We sometimes omit the subscript `$\R$' if it is clear from context.           
\item Functions $F:\R\di \R$ are represented by $\Phi^{1\di 1}$ mapping equal reals to equal reals, i.e. $(\forall x , y\in \R)(x=_{\R}y\di \Phi(x)=_{\R}\Phi(y))$.\label{rext}
\item The relation `$x\leq_{\tau}y$' is defined as in \eqref{aparth} but with `$\leq_{0}$' instead of `$=_{0}$'.  Binary sequences are denoted `$f^{1}, g^{1}\leq_{1}1$', but also `$f,g\in C$' or `$f, g\in 2^{\N}$'.  
%\item Sets of type $\rho$ objects are given by characteristic functions $X^{\rho\di 0}, Y^{\rho\di 0}, \dots$, i.e.\ 
%we write `$x^{\rho}\in X$' for $X(x)=_{0}0$, as in \cite{elkhuisje}.  Subsets of $\R$ are obtained by also requiring extensionality on the reals as in item \eqref{rext}.
\label{strijker}
\end{enumerate}
\end{defi}
%Secondly, we discuss the representation in $\RCAo$ of subsets of $\N$ and $2^{\N}$ \emph{in detail} in Section~\ref{kuiken}.  
We now discuss the issue of representations of real numbers.
\begin{rem}\label{forealsteve}\rm
First of all, introductory analysis courses often provide an explicit construction of $\R$ (perhaps in an appendix), while in practise one generally makes use of the axiomatic properties of $\R$, and not the explicit construction.  
Now, there are a number of different\footnote{The `early' constructions due to Dedekind (see e.g.\ \cite{kindke}; using cuts) and Cantor (see e.g.\ \cite{cant}; using Cauchy sequences) were both originally published in 1872.} such constructions: Tao uses Cauchy sequences in his text \cite{taoana1} and discusses decimal expansions in the Appendix \cite{taoana1}*{\S B}.  Hewitt-Stromberg also use Cauchy sequences in \cite{hestrong}*{\S5} and discuss Dedekind cuts in the exercises (\cite{hestrong}*{p.~46}).  Rudin uses Dedekind cuts in \cite{rudin} and mentions that Cauchy sequences yield the same result. 

\smallskip

Secondly, Definition \ref{keepintireal} is based on (fast-converging) Cauchy sequences, but Hirst has shown that over $\RCA_{0}$, individual real numbers can be converted between various representations (\cite{polahirst}).  Thus, the choice of representation in Definition \ref{keepintireal} does not really matter, even over $\RCA_{0}$.    
Moreover, the latter proves (\cite{simpson2}*{II.4.5}) that the real number system satisfies all the axioms of an Archimedian ordered field, i.e.\ we generally work with the latter axiomatic properties in RM, rather than with the representations (whatever they are).

\smallskip

Thirdly, converting sequences of real numbers between representations cannot always be done over $\RCA_{0}$, and $\WKL_{0}$ or $\ACA_{0}$ are sometimes needed, as also studed in \cite{polahirst}.
By the results in the latter, (fast-converging) Cauchy sequences are the `best' representation for the development of RM.  
\end{rem}
The previous remark deals with weak systems: $\exists^{2}$ from Section \ref{moresubsec} provides a uniform conversion facility between the various representations studied in \cite{polahirst}. 

\smallskip

Secondly, sets are represented by characteristic functions in Defintion \ref{openset}.  
Given $(\exists^{2})$ from Section \ref{moresubsec}, sets as in Definition \ref{openset} become `proper' characteristic function, only taking values `0' and `$1$'.  
For this and other reasons, we often assume the former axiom when dealing with sets.  
\bdefi[Sets in $\RCAo$]\label{openset}
We let $Y: \R \di \R$ represent subsets of $\R$ as follows: we write `$x \in Y$' for `$Y(x)>_{\R}0$'.   A set $Y\subseteq \R$ ‘open' if for every $x \in Y$, there is an open ball $B(x, r) \subset Y$ with $r^{0}>0$.  
A set $Y$ is called `closed' if the complement, denoted $Y^{c}=\{x\in \R: x\not \in Y \}$, is open. 
\edefi
\noindent

Hereafter, an `open set' refers to Definition \ref{openset}, while `RM-open set' refers to the RM-definition of open set as in \cite{simpson2}*{II.5.6}.
Now, one can effectively convert between RM-open sets and (RM-codes for) continuous characteristic functions (see \cite{simpson2}*{II.7.1}), i.e.\ our definition of (open) set is a generalisation of the RM-concept. 
We define countable sets as follows (see e.g.\ \cite{kunen}).  
\bdefi
A set $A\subset \R$ is \emph{countable} if there exists $Y:\R\di \N$ such that
\be\label{poemp}
(\forall x, y\in A)( Y(x)=_{0}Y(y)\di x=_{\R}y).
\ee
We say that $Y$ as in \eqref{poemp} is \emph{injective on $A$} or an \emph{injection from $A$ to $\N$}.
\edefi
In case $Y$ as in \eqref{poemp} is also surjective, we say that $A$ is \emph{strongly countable}, although we will
not study this concept in this paper. 

\smallskip

\noindent
Finally, for completeness, we list our notational conventions on finite sequences.  
\begin{nota}[Finite sequences]\label{skim}\rm
For $\rho=0,1$, we assume a dedicated type for `finite sequences of objects of type $\rho$', namely $\rho^{*}$.  Since the usual coding of pairs of numbers goes through in $\RCAo$, we shall not always distinguish between $0$ and $0^{*}$. 
Similarly, we do not always distinguish between `$s^{\rho}$' and `$\langle s^{\rho}\rangle$', where the former is `the object $s$ of type $\rho$', and the latter is `the sequence of type $\rho^{*}$ with only element $s^{\rho}$'.  The empty sequence for the type $\rho^{*}$ is `$\langle \rangle_{\rho}$', usually with the typing omitted.  

\smallskip

Furthermore, we denote by `$|s|=n$' the length of the finite sequence $s^{\rho^{*}}=\langle s_{0}^{\rho},s_{1}^{\rho},\dots,s_{n-1}^{\rho}\rangle$, where $|\langle\rangle|=0$, i.e.\ the empty sequence has length zero.
For sequences $s^{\rho^{*}}, t^{\rho^{*}}$, we denote by `$s*t$' the concatenation of $s$ and $t$, i.e.\ $(s*t)(i)=s(i)$ for $i<|s|$ and $(s*t)(j)=t(|s|-j)$ for $|s|\leq j< |s|+|t|$. For a sequence $s^{\rho^{*}}$, we define $\overline{s}N:=\langle s(0), s(1), \dots,  s(N-1)\rangle $ for $N^{0}<|s|$.  
For a sequence $\alpha^{0\di \rho}$, we also write $\overline{\alpha}N=\langle \alpha(0), \alpha(1),\dots, \alpha(N-1)\rangle$ for \emph{any} $N^{0}$.  By way of shorthand, 
$(\forall q^{\rho}\in Q^{\rho^{*}})A(q)$ abbreviates $(\forall i^{0}<|Q|)A(Q(i))$, which is (equivalent to) quantifier-free if $A$ is.   
\end{nota}
\subsubsection{Some higher-order systems and functionals}\label{moresubsec}
We introduce some functionals and axioms which constitute the counterparts of second-order arithmetic $\Z_{2}$, and some of the Big Five systems, in higher-order RM.
We use the `standard' formulation of these functionals as in \cite{kohlenbach2, dagsamIII}.  We are dealing with `conventional' comprehension, i.e.\ formula classes like $\Pi_{k}^{1}$ only boast first- and second-order parameters.

\smallskip
\noindent
First of all, $\ACA_{0}$ is readily derived from:
\begin{align}\label{mu}\tag{$\mu^{2}$}
(\exists \mu^{2})(\forall f^{1})\big[ (\exists n)(f(n)=0) \di [f(\mu(f))=0&\wedge (\forall i<\mu(f))f(i)\ne 0 ]\\
& \wedge [ (\forall n)(f(n)\ne0)\di   \mu(f)=0]    \big], \notag
\end{align}
and $\ACA_{0}^{\omega}\equiv\RCAo+(\mu^{2})$ proves the same sentences as $\ACA_{0}$ by \cite{hunterphd}*{Theorem~2.5}.   The (unique) functional $\mu^{2}$ in $(\mu^{2})$ is also called \emph{Feferman's $\mu$} (\cite{avi2}), 
and is clearly \emph{discontinuous} at $f=_{1}11\dots$; in fact, $(\mu^{2})$ is equivalent to the existence of $F:\R\di\R$ such that $F(x)=1$ if $x>_{\R}0$, and $0$ otherwise (\cite{kohlenbach2}*{\S3}), and to 
\be\label{muk}\tag{$\exists^{2}$}
(\exists \varphi^{2}\leq_{2}1)(\forall f^{1})\big[(\exists n)(f(n)=0) \asa \varphi(f)=0    \big]. 
\ee
\noindent
Secondly, $\FIVE$ is readily derived from the following sentence:
\be\tag{$\SS^{2}$}
(\exists\SS^{2}\leq_{2}1)(\forall f^{1})\big[  (\exists g^{1})(\forall n^{0})(f(\overline{g}n)=0)\asa \SS(f)=0  \big], 
\ee
and $\FIVE^{\omega}\equiv \RCAo+(\SS^{2})$ proves the same $\Pi_{3}^{1}$-sentences as $\FIVE$ by \cite{yamayamaharehare}*{Theorem 2.2}.   The (unique) functional $\SS^{2}$ in $(\SS^{2})$ is also called \emph{the Suslin functional} (\cite{kohlenbach2}).
By definition, the Suslin functional $\SS^{2}$ can decide whether a $\Sigma_{1}^{1}$-formula \emph{in normal form}, i.e.\ as in the left-hand side of $(\SS^{2})$, is true or false.   We similarly define the functional $\SS_{k}^{2}$ which decides the truth or falsity of $\Sigma_{k}^{1}$-formulas in normal form; we also define 
the system $\SIXK$ as $\RCAo+(\SS_{k}^{2})$, where  $(\SS_{k}^{2})$ expresses that $\SS_{k}^{2}$ exists.  Note that we allow formulas with \emph{function} parameters, but \textbf{not} \emph{functionals} here.
In fact, Gandy's \emph{Superjump} (\cite{supergandy}) constitutes a way of extending $\FIVE^{\omega}$ to parameters of type 2; see the discussion in \cite{dagsamV}*{\S2.3}.

\smallskip

\noindent
Thirdly, full second-order arithmetic $\Z_{2}$ is readily derived from $\cup_{k}\SIXK$, or from:
\be\tag{$\exists^{3}$}
(\exists E^{3}\leq_{3}1)(\forall Y^{2})\big[  (\exists f^{1})Y(f)=0\asa E(Y)=0  \big], 
\ee
and we therefore define $\Z_{2}^{\Omega}\equiv \RCAo+(\exists^{3})$ and $\Z_{2}^\omega\equiv \cup_{k}\SIXK$, which are conservative over $\Z_{2}$ by \cite{hunterphd}*{Cor.\ 2.6}. 
Despite this close connection, $\Z_{2}^{\omega}$ and $\Z_{2}^{\Omega}$ can behave quite differently, as discussed in e.g.\ \cite{dagsamIII}*{\S2.2}.   The functional from $(\exists^{3})$ is also called `$\exists^{3}$', and we use the same convention for other functionals.  

\smallskip

Fourth, recall that the Heine-Borel theorem (aka \emph{Cousin's lemma} \cite{cousin1}*{p.\ 22}) states the existence of a finite sub-cover for an open cover of certain spaces. 
Now, a functional $\Psi:\R\di \R^{+}$ gives rise to the \emph{canonical} cover $\cup_{x\in I} I_{x}^{\Psi}$ for $I\equiv [0,1]$, where $I_{x}^{\Psi}$ is the open interval $(x-\Psi(x), x+\Psi(x))$.  
Hence, the uncountable cover $\cup_{x\in I} I_{x}^{\Psi}$ has a finite sub-cover by the Heine-Borel theorem; in symbols:
\begin{princ}[$\HBU$]\label{drifty}
$(\forall \Psi:\R\di \R^{+})(\exists  y_{1}, \dots, y_{k}\in I){(\forall x\in I)}(\exists i\leq k)(x\in I_{y_{i}}^{\Psi})$.
\end{princ}
By the results in \cite{dagsamIII, dagsamV}, $\Z_{2}^{\Omega}$ proves $\HBU$ but $\Z_{2}^{\omega}+\QFAC^{0,1}$ cannot, 
%Hence, the Heine-Borel theorem for uncountable covers as in $\HBU$ falls \emph{far} outside of the Big Five of RM, as noted at the end of Section \ref{RM}.  
and many basic properties of the \emph{gauge integral} (\cite{zwette, mullingitover}) are equivalent to $\HBU$.  We have also studied the \emph{Lindel\"of lemma} for $\R$ in \cites{dagsamIII, dagsamV}. 
%By Remark \ref{kloti}, we may drop the requirement that $\Psi$ in $\HBU$ needs to be extensional on the reals, i.e.\ $\Psi$ does not have to satisfy \eqref{RE} from Definition \ref{keepintireal}.
\begin{princ}[$\LIN$]\label{drifty2}
$(\forall \Psi:\R\di \R^{+})(\exists  (y_{n})_{n\in \N}){(\forall x\in \R)}(\exists k\in \N)(x\in I_{y_{k}}^{\Psi})$.
\end{princ}
Furthermore, since Cantor space (denoted $C$ or $2^{\N}$) is homeomorphic to a closed subset of $[0,1]$, the former inherits the same property.  
In particular, for any $G^{2}$, the corresponding `canonical cover' of $2^{\N}$ is $\cup_{f\in 2^{\N}}[\overline{f}G(f)]$ where $[\sigma^{0^{*}}]$ is the set of all binary extensions of $\sigma$.  By compactness, there are $ f_0 , \ldots , f_n\in 2^{\N}$ such that $\cup_{i\leq n}[\bar f_{i} G(f_i)]$ still covers $2^{\N}$.  By \cite{dagsamIII}*{Theorem 3.3}, $\HBU$ is equivalent to the same compactness property for $C$, as follows:
\be\tag{$\HBU_{\c}$}
(\forall G^{2})(\exists  f_{1}, \dots, f_{k} \in 2^{\N} ){(\forall f \in 2^{\N})}(\exists i\leq k)(f\in [\overline{f_{i}}G(f_{i})]).
\ee
 On a technical note, when we say `finite sub-cover', we mean the set of the associated neighbourhoods, not `just' their union.    
We now introduce the specification $\SFF(\Theta)$ for a functional $\Theta^{2\di 1^{*}}$ which computes a finite sequence as in $\HBU_{\c}$.  
We refer to such a functional $\Theta$ as a \emph{realiser} for the compactness of Cantor space, and simplify its type to `$3$'.  % to improve readability.
%\bdefi\label{dodier}
%The formula $\SFF(\Theta)$ is as follows for $\Theta^{2\di 1^{*}}$:
Clearly, there is no unique such $\Theta$: just add new sequences to $\Theta(G)$.
\be\tag{$\SFF(\Theta)$}
(\forall G^{2})(\forall f^{1}\leq_{1}1)(\exists g\in \Theta(G))(f\in [\overline{g}G(g)]).
\ee
%\edefi
Any functional $\Theta$ satisfying $\SFF(\Theta)$ is called a \emph{special fan functional} or simply a \emph{$\Theta$-functional}.
% nonetheless, 
%we have in the past referred to any $\Theta$ satisfying $\SFF(\Theta)$ as `the' \emph{special fan functional} $\Theta$, and we will continue this abuse of language.  
As to its provenance, $\Theta$-functionals were introduced as part of the study of the \emph{Gandy-Hyland functional} in \cite{samGH}*{\S2} via a different definition.  
These are identical up to a term of G\"odel's $T$ of low complexity by \cite{dagsamII}*{Theorem~2.6}.  

\smallskip

Finally, we have studied countable sets in \cites{dagsamXI, dagsamXII, dagsamX}, in the guise of the following. 
\begin{princ}[$\cocode_{0}$]
For any countable $A\subset [0,1]$, there is a sequence $(x_{n})_{n\in \N}$ that contains all elements of $A$.
\end{princ}
This principle is `explosive' in the sense that $\FIVE^{\omega}+\cocode_{0}$ proves $\SIX$, while $\FIVE^{\omega}$ is $\Pi_{3}^{1}$-conservative over $\FIVE$ (see \cite{dagsamXI, dagsamX}).  

\subsection{Higher-order computability}\label{prelim2}
As some of our main results are part of computability theory, we make our notion of `computability' precise as follows.  
\begin{enumerate}
\item[(I)] We adopt $\ZFC$, i.e.\ Zermelo-Fraenkel set theory with the Axiom of Choice, as the official metatheory for all results, unless explicitly stated otherwise.
\item[(II)] We adopt Kleene's notion of \emph{higher-order computation} as given by his nine clauses S1-S9 (see \cites{longmann, Sacks.high}) as our official notion of `computable'.
\end{enumerate}
A thorough introduction to Kleene computability theory may be found in \cite{longmann}.
We do recall an important notion from the latter.  
\begin{rem}[Normal and non-normal mathematics]\label{dorf}\rm
The distinction between `normal' and `non-normal' mathematics is based on the following definition.
\begin{center}
For $n\geq 2$, a functional of type $n$ is called \emph{normal} if it computes Kleene's $\exists^{n}$ following S1-S9, and \emph{non-normal} otherwise. (\cite{longmann}*{\S5.4})  
\end{center}
Similarly, we call a statement about type $n$ objects ($n\geq 2$) \emph{normal} if it implies the existence of $\exists^{n}$ over Kohlenbach's base theory from Section \ref{prelim1}, and \emph{non-normal} otherwise.  
We also use `\emph{strongly} non-normal' for type $3$ functionals that do not compute $\exists^{3}$ \emph{relative to $\exists^{2}$}.  
Note that by \cite{kohlenbach2}*{\S3}, $(\exists^{2})$ is equivalent to the existence of a discontinuous function on $\R$.

\smallskip

Historically, higher-order computability theory and higher-order RM have mostly been focused on the normal world.  
Recently, the authors have identified $\HBU$ and $\Theta$-functionals as interesting parts of the \emph{non-normal} world (\cite{dagsam, dagsamII,dagsamIII}). 
Since $\HBU$ can be formulated in third-order arithmetic, `$\HBU$ is non-normal' means that $\HBU$ does not prove $(\exists^{2})$ in this case.  
The associated $\Theta$-functionals are fourth-order and `a given $\Theta$-functional is non-normal' thus means that it does not compute $\exists^{3}$.  
The same holds for $\WHBU$ (see Definition \ref{ohyee}) and the associated $\Lambda$-functionals (see Definition~\ref{lambdagvd}). 
The uncountability of $\R$, when formulated using injections or bijections to $\N$, is similarly non-normal (see \cite{dagsamX, dagsamXI, dagsamXII}).  

\smallskip

However, it is an empirical observation that the above non-normal theorems and functionals, which are intuitively `weak', are classified as `hard to prove' and the associated functionals as `hard to compute' \emph{relative to the normal scale based on comprehension and discontinuous functionals}: in each case $\Z_{2}^{\omega}$ cannot prove the theorem and no $\SS_{k}^{2}$ can compute a realiser, while $\Z_{2}^{\Omega}$ and $\exists^{3}$ suffice.  
In this way, the normal scale gives intuitively \emph{weak} non-normal theorems and functionals \emph{the same} classification, namely \emph{rather strong}.  In this light, the normal scale  seems unsuitable for analysing non-normal theorems and functionals.  
Thus, the need for the development of the non-normal scale arises, which is the topic of this paper, and also of \cite{dagsam, dagsamII, dagsamIII,dagsamVII, dagsamV, dagsamIX, dagsamX, dagsamXI, dagsamXII}.  
Here, Theorem \ref{theorem.lambda} is a `milestone' result from computability theory while Theorem \ref{Frink} is a milestone in RM, where the non-normal nature of $\WHBU$ follows from that of $\HBU$.
% of the components. 

\smallskip

Finally, the importance of the of `normal versus non-normal' distinction was only really understood by the authors after the completion of \cites{dagsamX, samph}.
\end{rem}
\section{Reverse Mathematics and $\WHBU$}\label{RMM}
\subsection{Introduction}
In this section, we study the RM of measure theory,  $\WHBU$ in particular, as summarised by the following list.  
\begin{itemize}
\item In Section \ref{introk}, we introduce $\WHBU$, a version of the Vitali covering theorem that is a direct weakening of $\HBU$ from Section \ref{moresubsec}.  We investigate generalisations of $\WHBU$ akin to those studied in second-order RM.  
\item We show that various instances of Littlewood's three principles (including Lusin's and Egorov's theorems) imply $\WHBU$ (Section \ref{kolp}).
%Our choice of framework is motivated in Section \ref{motig}, while we sketch the system from \cite{elkhuisje} in Section \ref{kuiken}.
%In Section~\ref{hintro}, we introduce the central principle $\WHBU$ expressing \emph{the Vitali covering theorem}
\item We study Kreuzer's measure theory \cite{elkhuisje} (Section \ref{onderkruipers}). We derive the Egorov's theorem but show that the Heine-Borel theorem cannot be proved.  
\item We show that $\WHBU$ also occurs in an alternative (very different) approach to the Lebesgue integral (Section \ref{GL}), namely the \emph{gauge integral}.  
A similar result for the Riemann integral is obtained.
\end{itemize}
 %we show in Section \ref{grm} that this theorem is essential to convergence theorems for the Riemann integral involving \emph{nets}; the latter constitute the generalisation of the notion of sequence to (possibly) uncountable index sets.  
Thus, $\WHBU$ is shown to arise naturally in different approaches to measure theory, i.e.\ our results can be said to be independent of the particular framework.  
Regarding the third item, we \emph{could} obtain equivalences involving $\WHBU$, but this would require a base theory beyond the scope of this paper. 

\smallskip

Finally, we discuss an important convention as to the meaning of the Lebesgue measure.  
In a nutshell, except in Section \ref{onderkruipers}, we interpret the Lebesgue measure in a `virtual' or `comparative' sense similar to the approach in second-order RM.  
\begin{rem}[A measure by any other name]\label{dreft}\rm
First of all, Lebesgue measure theory can be developed in second-order RM (see e.g.\ \cite{simpson2}*{X.1}).  
However, the Lebesgue measure is defined via a supremum (see Definition \cite{simpson2}*{X.1.2}) that need not always exist in weak systems like $\RCA_{0}$.  Nonetheless, $\L_{2}$-formulas
like e.g.\ 
\be\label{exakt}
\text{\emph{the Lebesgue measure of a given open set $U$ is at most $1$}}
\ee
always makes sense, even in $\RCA_{0}$.  Indeed, \eqref{exakt} essentially expresses that any continuous approximation (from below) of the characteristic function of $U$ will have Riemann integral at most $1$.
This `comparative' or `virtual' meaning of \eqref{exakt} is described in detail in \cite{simpson2}*{p.\ 392}.  

\smallskip

We will always interpret comparative statements involving the Lebesgue measure in this `virtual' or `comparative' way (except Section \ref{onderkruipers}).  
Now, the usual definition of the Lebesgue measure on $\R$ is as follows (see e.g.\ \cite{taomes}):
\be\label{LM}
{\textstyle \lambda (E)=\inf \left\{\sum _{k} |I_{k}|:{(I_{k})_{k\in \mathbb {N} }}{\text{ is a sequence of intervals and }}E\subset \bigcup _{k}I_{k}\right\},}
\ee
in case this infimum exists.  Due to the quantification over sequences, we observe that $\Z_{2}^{\Omega}$ can always define $\lambda(E)$, assuming this infimum exists.  
However, statements like `$\lambda(E)>1$' can be interpreted in the aforementioned comparative sense, for which the associated infimum need not exist.  
To be absolutely clear, the formula `$\lambda(E)\geq \frac{1}{2}$' is purely symbolic and short for
\[\textstyle
\text{\emph{For a sequence of open intervals $ (I_{k})_{k\in \mathbb {N}} $ with $E\subset \bigcup _{k\in \N}I_{k}$, we have $\frac{1}{2}\leq\sum _{k=1}^{\infty } |I_{k}|$}}.
\]
In this way, comparative statements about the Lebesgue measure of arbitrary sets make sense in $\RCAo$, even if the infimum in \eqref{LM} does not always exist.  
It is important to note that this convention essentially `hard-wires' the first Littlewood principle into the definition of the Lebesgue measure.  
Based on the above convention, we say that a set $E \subseteq [0,1]$ is \emph{measurable} if 
\[
\lambda(E) + \lambda([0,1] \setminus E) \leq 1 .
\]
Similarly, a \emph{measurable} function $f:\R\di \R$ is defined as usual, namely as saying that for all $t\in \R$, the set $\{x\in \R : f(x)>_{\R}t\}$ is measurable.  
Since the latter set or a union $\cup_{n\in \N}A_{n}$ only exists as a set given $\exists^{2}$, we will often work over $\ACAo$.

\smallskip

Finally, the Lebesgue measure for open sets (in the sense of second-order RM) exists in $\ACA_{0}$, while the latter system suffices for a general treatment of the subject (see \cite{simpson2}*{X.1}).  
We let $(\blambda_{\open})$ be the non-normal statement that there exists $\blambda:(\R\di \R)\di \R$ such that for open $E\subset [0,1]$, $\blambda(E)$ equals the infimum as in \eqref{LM}.
Below, we show that the fragment $(\blambda_{\open})$ of the Lebesgue integral allows for a generalisation of $\WHBU$ to general coverings of arbitrary closed sets. 
\end{rem}

\subsection{The Vitali covering theorem and $\WHBU$}\label{introk}

\subsubsection{Introduction}\label{hintro}
In this section, we introduce $\WHBU$, a version of the Vitali covering theorem that is a direct weakening of $\HBU$ from Section \ref{moresubsec}.
We also establish some basic properties of $\WHBU$ in Section \ref{klopje}.

\smallskip

As to notation, recall that $\big(x-{\Psi(x)}, x+{\Psi(x)}\big)$ is denoted as $I_{x}^{\Psi}$ or $B(x, \Psi(x))$ for $\Psi:[0,1]\di \R^{+}$, while $\cup_{x\in [0,1]}B(x, \Psi(x))$ is called the \emph{canonical cover} of the unit interval generated by $\Psi$.
Also recall the convention concerning the Lebesgue measure from Remark \ref{dreft}.  Now consider the following:
\begin{princ}[$\WHBU$]\label{ohyee}
For $ \Psi:[0,1]\di \R^{+}$ and $\eps>_{\R}0$, there are $y_{0}, \dots, y_{n}\in [0,1]$ such that the measure of $\cup_{i\leq n} I_{y_{i}}^{\Psi} $ is at least $1-\eps$.
%where $I_{x}^{\Psi}$ is the open interval $\big(x-\frac{1}{\Psi(x)}, x+\frac{1}{\Psi(x)}\big)$. %$(x-\frac{1}{\Psi(x)+1}, x+\frac{1}{\Psi(x)+1})$.  
\end{princ}
As suggested by its name, $\WHBU$ is a weakening of $\HBU$. 
In fact, $\HBU$ is to $\WHBU$ what $\WKL$ is to $\WWKL$.  
The latter is \emph{weak weak K\"onig's lemma} and may be found in \cite{simpson2}*{X.1}, while the $\ECF$-translation converts the former two into the latter two.  
Now, $\WHBU$ constitutes the essence of \emph{Vitali's covering theorem} as follows, a version of which was introduced in 1907 (\cite{vitaliorg}).  
\begin{quote}
If $\mathcal{I}$ is a Vitali cover of $E\subset I$, then there is a sequence of disjoint intervals $I_{n}$ in $\mathcal I$ such that $E\setminus \cup_{n\in \N}I_{n}$ has measure zero.
\end{quote}
Indeed, a \emph{Vitali cover} of a set is an open cover in which every element of the set can be covered by an open set of \emph{arbitrary small} size (\cite{royden1}*{Ch.\ 5.1}).    
Vitali's covering theorem for countable coverings is equivalent, over $\RCA_0$, to $\WWKL_0$ by \cite{simpson2}*{X.1.13}.  Working in $\ACAo+\QFAC^{0,1}$, we may apply the latter to $\WHBU$ to obtain a countable Vitali sub-cover of a set of measure 1 of a given\footnote{For $\Psi:[0,1]\di \R^{+}$, one defines $\Psi_{k}:([0,1]\times \N)\di\R^{+} $ as $\Psi_{k}(x):= \frac{\Psi(k)}{2^{k}}$, which yields a `canonical' Vitali cover $\cup_{k\in \N}\cup_{x\in [0,1]}B(x, \Psi_{k}(x))$ of $[0,1]$ generated by $\Psi$.\label{lightfoot}} Vitali cover. 
In this way Vitali's covering theorem is provable from $\WHBU$, while proving the latter from the former is an easy exercise. 
As discussed in \cite{opborrelen}*{Note, p.\ 50-51}, Borel actually proves the (countable) Heine-Borel theorem to justify his use of the following lemma: 
\begin{center}
\emph{If $1>_{\R}\sum_{n=0}^{\infty}|a_{n}-b_{n}|$, then $\cup_{n\in \N}(a_{n}, b_{n})$ does not cover $[0,1]$}.  
\end{center}
However, the latter is equivalent to $\WWKL$ by \cite{simpson2}*{X.1.9}, which provides some historical motivation and context.  

\smallskip

Like for $\HBU$, $\Z_{2}^{\Omega}$ proves $\WHBU$, but $\Z_{2}^{\omega}$ cannot\footnote{The model $\mathcal{M}$ from \cite{dagsamV}*{\S4.1} satisfies $\SIXK+\QFAC^{0,1}$, but not $\WHBU$.  
This model is obtained from the proof that a realiser for $\WHBU$ is not computable in any type two functional. 
}.  
Hence, $\WHBU$ is quite hard to prove (in terms of conventional comprehension), and the finite sequence of reals in $\WHBU$ is similarly hard to compute: 
a $\Lambda$-functional is (equivalently) defined in Section \ref{CTM} by saying that $\Lambda(\Psi, \eps)$ computes $\langle y_{1}, \dots, y_{n}\rangle$ as in $\WHBU$.   
Like for $\Theta$-functionals, there is no unique such $\Lambda$-functional. 
Moreover, no type two functional can compute a $\Lambda$-functional (see \cite{dagsam, dagsamII}), which includes the comprehension functionals $\SS_{k}^{2}$ from Section \ref{moresubsec}.
These hardness properties do not disappear if we restrict $\Psi:[0,1]\di \R^{+}$ in $\WHBU$ to e.g.\ Borel\footnote{The proof of \cite{samcie22}*{Theorem 16 and Cor.\ 17} goes through with trivial modification, i.e.\ $\WHBU$ restricted to Baire 2 or semi-continuous functions implies that \emph{there is no injection from $2^{\N}$ to $\N$}.  The latter statement is however not provable in $\Z_{2}^{\omega}+\QFAC^{0,1}$ by \cite{dagsamXII}*{Theorem 3.2}.\label{floeker}} or semi-continuous functions.  

\subsubsection{Generalisations}\label{klopje}
We study the following rather straightforward generalisations of $\WHBU$, the analogues of which have been studied in second-order RM, namely \cite{simpson2}*{IV.1.6} and \cite{brownphd}*{Lemma 3.13}. 
\begin{enumerate}
\item[(a)] We replace the `interval' covering $\cup_{x\in [0,1]}B(x, \Psi(x))$ by a `general' covering $\cup_{x\in [0,1]}O_{x}$, only assuming that the open $O_{x}$ contains $x$ for any $x\in [0,1]$.
\item[(b)] We replace the covering $\cup_{x\in [0,1]}B(x, \Psi(x))$ of the unit interval $[0,1]$ by a covering $\cup_{x\in E}B(x, \Psi(x))$ of arbitrary $E\subset [0,1]$.
\end{enumerate}
As will become clear, these generalisations follow from $\WHBU$ and basic properties of the Lebesgue measure. 
We show that the same generalisations for $\HBU$ and the Lindel\"of lemma are much stronger than the original principles, highlighting 
a fundamental difference.  This kind of behaviour is also discussed in \cite{samcie21}.

\smallskip

First of all, motivated by item (a) right above, we define the following notion. 
\bdefi
A \emph{general} open covering of $X\subset \R$ is a mapping $\lambda x.O_{x}:\R\di (\R\di \R)$ such that $O_{x}$ is an open set containing $x$, for any $x\in X$.
\edefi
The following result shows that a fragment of the Lebesgue measure already significantly generalises $\WHBU$.  
We note that Kreuzer's measure theory from Section~\ref{onderkruipers} proves $(\blambda_{\open})$ (see Remark \ref{dreft}) and $\WHBU$ (see Theorems \ref{egotripke} and \ref{egotripke2}).
\begin{thm}[$\ACAo+(\blambda_{\open})$]\label{kag} The following are equivalent.
\begin{itemize}
\item $\WHBU$
\item For a general open covering $\lambda x.O_{x}$ and $\eps>0$,  there are $x_{0}, \dots, x_{k}\in [0,1]$ such that the measure of $\cup_{i\leq k}O_{x_{i}}$ is $>1-\eps$.
\end{itemize}
\end{thm}
\begin{proof}
Let $\lambda y.O_{y}$ be as in the theorem and note that the following follows by definition for any  $x\in [0,1]$:
\be\textstyle\label{gingi}
(\exists n\in \N)(\blambda(O_{x}\cap B(x, \frac{1}{2^{n+1}})  )=\frac{1}{2^{n}}),
\ee
where we note that the set in \eqref{gingi} is open.
Let $B(x)$ be the least $n$ as in \eqref{gingi}.
Then $\Psi(x):=2^{-(B(x)+2)}$ yields a canonical covering of $[0,1]$.  For $\eps>0$, $\WHBU$ yields $x_{0}, \dots, x_{k}\in [0,1]$ such that $\cup_{i \leq k} B(x_i, \Psi(x_{i}))$ has measure at least $1-\eps$.
Clearly, $x_{0}, \dots, x_{k}\in [0,1]$ is also such that the measure of $\cup_{i\leq k}O_{x_{i}}$ is at least $1-\eps$.  Indeed, while $B(x, \Psi(x))$ may not be a subset of $O_{x}$, we do have that the measure of $O_{x}$ is at least that of $B(x, \Psi(x))$.
\end{proof}
In short, $\WHBU$ `bootstraps' itself to general open coverings, thanks to (a fragment of) the Lebesgue measure.  
By contrast, the Lindel\"of lemma for general open coverings implies the `explosive' principle $\cocode_{0}$ from Section \ref{moresubsec}.
\begin{thm}[$\ACAo$]\label{kassle}
The Lindel\"of lemma for general open coverings of $\R$ implies $\cocode_{0}$. 
\end{thm}
\begin{proof}
Fix $A\subset [0,1]$ and $Y:[0,1]\di \N$ injective on $A$.  
Note that $\mu^{2}$ can enumerate all rationals in $A$, i.e.\ we may assume $A\cap \Q=\emptyset$.
Define the set $B$ as follows:
\[
y\in B \asa (\exists n\in \N)\big[ y\in [2n+1, 2n+2] \wedge (y-(2n+1)) \in A \wedge Y(y-(2n+1))=n     \big]. 
\]
Clearly, if we can enumerate the reals in $B$, we can enumerate the reals in $A$.
Now define the following general open covering on $[n, n+1]$ (for any $n\in \N$):
\be\label{maarho2}
O_{x}:=
\begin{cases}
\textup{the set $(n-1, n+3)\setminus B$} & \textup{if }  x\not \in B\wedge x\in [n, n+1)\\
%\textup{the set $(2n+\frac{1}{4}, 2n+\frac{7}{4})\setminus B$} & x\not \in B\wedge x\in [2n+\frac{1}{2}, 2n+\frac{3}{2}]\\
%\textup{the set $(n, n)\setminus B$} & x\not \in B\wedge x\in [n+1, n+2]\\
\textup{the interval $B(x, d_{n}(x))$} & \textup{if } x\in B\wedge x\in (n, n+1),
\end{cases}
\ee
where $d_{n}(x)=\min\big( \frac{|x-n|}{2}, \frac{|x-(n+1)|}{2}  \big)$ in case $x\in (n, n+1]$. 
Clearly, $O_{x}$ is an open set such that $x\in O_{x}$ for any $x\geq 0$ and the extension to $\R$ is trivial.
Now let $(x_{n})_{n\in \N}$ be such that $\cup_{n\in \N}O_{x_{n}}$ covers $\R$.  
Use $\mu^{2}$ to remove all elements in the sequence not in $B$.  
The resulting sequence lists all reals in $B$ as `by definition' reals in $B$ are not covered by the set in the first case of \eqref{maarho2}.
Thus, we are done. 
\end{proof}
The proof of Theorem \ref{kag} also shows that any $\Lambda$-functional is readily generalised to general open coverings if we have access to the Lebesgue measure (for open sets).
By Theorem \ref{kank}, $\Theta$-functionals for general open coverings are quite explosive. 
\begin{princ}[$Z$-functional]\label{rampant}
A functional $Z^{3}$ is called a $Z$-functional if $Z(\lambda x.O_{x})$ is a finite sequence $x_{0}, \dots, x_{k}\in [0,1]$ such that $[0,1]$ is covered by $\cup_{i\leq k}O_{x_{i}}$ for every general open covering $\lambda x.O_{x}$ of $[0,1]$.
\end{princ}
Let $\Omega_{\textup{b}}$ be the partial functional that returns $0$ if the input set is the empty set and returns $1$ in if the input set is a singleton.  
Basic as this functional may seem, together with $\SS^{2}$, it computes $\SS_{2}^{2}$ (see \cite{dagsamXIII, dagsamXII}).  Moreover, we have the following. 
\begin{thm}\label{kank}
Together with $\exists^{2}$, a $Z$-functional computes $\Omega_{\textup{b}}$.  
\end{thm}
\begin{proof}
Let $X\subset [0,1]$ be finite and define the following general open covering:
\be\label{maarho}
O_{x}:=
\begin{cases}
\textup{the set $(-2,2)\setminus X$} & \textup{ if } x\not \in X\\
\textup{the interval $(-2, 2)$} & \textup{ otherwise },
\end{cases}
\ee
%where $d(x)=\min\big( \frac{|x-1|}{2}, \frac{|x-0|}{2}  \big)$.  
Clearly, $O_{x}$ is an open set such that $x\in O_{x}$ for any $x\in [0,1]$.
%Now define $\Omega_\textup{b}(X) = 1$ if $0 \in X$ or $1 \in X$; in case $0,1\not \in X$, 
Now consider $Z(\lambda x.O_{x})=(x_{0}, \dots, x_{k})$ and note that in case for all $i\leq k$, we have $x_{i}\not \in X$, the set $X$ must be empty.
Hence, define $\Omega_{\textup{b}}(X)$ as $0$ in this case, and $1$ otherwise. 
\end{proof}
Secondly, motivated by item (b) from the beginning of this section, we study coverings of arbitrary sets, as in the following principle.
\begin{princ}[$\WHBU^{+}$]
For $\Psi:[0,1]\di \R^{+}$, $\eps>0$, and non-empty $E\subset [0,1]$ with measure $ e\in [0,1]$, there are $x_{0}, \dots, x_{k}\in E$ such that the measure of $\cup_{i\leq k}B(x_{i}, \Psi(x_{i}))$ is at least $e-\eps$
\end{princ}
We let $\WHBU_{\RM}^{+}$ be $\WHBU^{+}$ restricted to RM-closed sets $E$.  We now show that $\WHBU$ can be `bootstrapped' as follows. 
\begin{thm}[$\ACAo$]\label{fagan} We have $\WHBU\asa \WHBU^{+}_{\RM}\asa \WHBU^{+}$.
%For $\Psi:[0,1]\di \R^{+}$, $\eps>0$, and non-empty R2-closed $E\subset [0,1]$ with measure $ e\in [0,1]$, there are $x_{0}, \dots, x_{k}\in E$ such that the measure of $\cup_{i\leq k}B(x_{i}, \Psi(x_{i}))$ is at least $e-\eps$
\end{thm}
\begin{proof}
We first prove $\WHBU\di \WHBU^{+}_{\RM}$.
Fix $\Psi:[0,1]\di \R^{+}$, $\eps>0$, and RM-closed $E\subset [0,1]$, where the latter is represented by the complement of $\cup_{n\in \N}(a_{n}, b_{n})$.  
Define $\Phi:[0,1]\di \R^{+}$ as follows:
\be\label{dagwim}
\Phi(x):= 
\begin{cases}
\Psi(x) & \textup{in case $x\in E$}\\
D(x) & \textup{in case $x\not \in E$}
\end{cases},
\ee
where $D(x)$ is $r\in \Q^{+}$ such that $B(x, r)\subset (a_{m}, b_{m})$ in case $m\in \N$ is the least natural such that $x\in (a_{m},b_{m}) $, and $0$ otherwise.  
%Here, $B(x)$ is the least $n\in \N$ such that $(\forall q\in \Q)(  q\in B(x, \frac{1}{2^{n}}) \di q\in I\setminus E )$.
%Note that $B(x)$ is non-trivial as $I\setminus E$ is open. 
 Apply $\WHBU$ to $\cup_{x\in [0,1]}B(x, \Phi(x))$ to obtain $x_{0}, \dots, x_{k}\in [0,1]$ such that the measure of $\cup_{i\leq k}B(x_{i}, \Phi(x_{i}))$ is at least $1-\eps$.
Let $y_{0}, \dots, y_{m}$ be those $x_{i} $ for $i\leq k$ that are in $E$, and let $z_{0}, \dots, z_{n}$ be the remaining ones.  
Since all sets involved are intervals, we may use the usual (second-order) Lebesgue measure $\lambda$ and obtain:
\begin{eqnarray*}
1-\eps <\lambda(\cup_{i\leq k}B(x_{i}, \Phi(x_{i})))
&\leq \lambda(\cup_{j\leq m}B(y_{j}, \Phi(y_{j})))+\lambda(\cup_{i\leq  n }B(z_{i}, \Phi(z_{i}))) \\
&= \lambda(\cup_{j\leq m}B(y_{j}, \Psi(y_{j})))+\lambda(\cup_{i\leq  n }B(z_{i}, \Phi(z_{i})))\\
&\leq \lambda(\cup_{j\leq m}B(y_{j}, \Psi(y_{j})))+(1-e), 
\end{eqnarray*}
which shows that $y_{0}, \dots, y_{m}\in E$ are as required by $\WHBU_{\RM}^{+}$.  

\smallskip

To show that $\WHBU^{+}_{\RM}\di \WHBU^{+}$, fix $\Psi:[0,1]\di \R^{+}$, $\eps>0$, and $E\subset [0,1]$ with measure $\lambda(E)=e$.
Then $[0,1]\setminus E$ has measure $1-e$ and by definition there is a sequence of open intervals $(I_{n})_{n\in \N}$ such that $\cup_{n\in \N}I_{n}$ covers $[0,1]\setminus E$ and has measure at most $(1-e)+\eps/2$.
Then $C:=[0,1]\setminus \cup_{n\in \N}I_{n}$ is RM-closed, satisfies $C\subset E$, and has measure at least $e-\eps/2$.  Now apply $\WHBU_{\RM}^{+}$ for $C$ and $\eps/2$.   
%In case the measure of $\cup_{i\leq k}B(y_{i}, \Psi(y_{i}))$ is below $\blambda(E)-\eps$, then $$
\end{proof} 
We observe that the second part of the proof only goes through because our comparative interpretation of the Lebesgue measure essentially hard-codes Littlewood's first principle. 
The first part of the proof of the theorem is `effective' and shows that a $\Lambda$-functional is readily generalised to RM-closed sets.
To generalise $\Lambda$-functionals to coverings of closed sets, let $\WHBU_{\closed}^{+}$ be the restriction of $\WHBU^{+}$ to closed sets (Definition \ref{openset}). 
%One could obtain similar results for $\HBU$, including equivalences between the R2-property and the Heine-Borel property,
%but this would require a stronger base theory, where a candidate for the latter is studied in \cite{dagsamIX}.  
\begin{cor}[$\ACAo+(\blambda_{\open})$]\label{carry}
We have $\WHBU\di \WHBU_{\closed}^{+}$.
\end{cor}
\begin{proof}
Recall $B:\R\di \N$ as defined from \eqref{gingi}.  In the proof of Theorem \ref{fagan}, modify the functional $\Phi$ from \eqref{dagwim} as follows: $\Phi(x)$ is $\Psi(x)$ in case $x\in E$, and $\frac{1}{2^{B(x)}}$ otherwise.  
The first part of the proof of the theorem now goes through. 
\end{proof}
The proof of the corollary shows that a $\Lambda$-functional is readily generalised to any closed set, assuming the Lebesgue measure $\blambda$ as in $(\blambda_{\open})$.
However, if we consider Vitali covers as in Footnote \ref{lightfoot}, then we can shrink the measure of $\cup_{i\leq k}B(x_{i}, \Psi(x_{i}))$ from Theorem \ref{fagan} below $e+\eps$ if necessary.   
In this way, a $\Lambda$-functional generalised to general open coverings of closed sets is computationally equivalent to the combination of: a `standard' $\Lambda$-functional and the Lebesgue measure $\blambda$ as in $(\blambda_{\open})$.
The general case, involving coverings of arbitrary sets, goes through \emph{mutatis mutandis}.

\smallskip

Next, we show that the same generalisations for the Lindel\"of lemma are quite powerful.  
A `fixed radius interval covering' of $E$ is $\cup_{x\in E}B(x, \eps)$ for some $\eps>0$.
\begin{thm}[$\ACAo$]\label{koti}
The principle $\cocode_{0}$ follows from the Lindel\"of lemma for fixed radius interval coverings of closed sets in $\R$.
\end{thm}
\begin{proof}
Consider the sets $A, B$ from the proof of Theorem \ref{kassle}.  Now apply the Lindel\"of lemma from the theorem to $\cup_{x\in B} (x-\frac{1}{2}, x+\frac{1}{2})$, which readily yields of enumeration on $B$, and hence of $A$.  
\end{proof}
On a historical note, Lindel\"of in \cite{blindeloef} formulates his lemma for general open coverings of \emph{any} set, 
while the Heine-Borel theorem for open coverings of closed sets may be found in e.g.\ \cite{lennes}.

\smallskip

Finally, Theorem \ref{koti} deals with coverings consisting of open intervals with a fixed radius.
Trivial as such coverings many seem, they play an important role in Section \ref{kolp} in the form of the following principle, which follows from $\WHBU$.  
\begin{princ}[$\WHBU^{-}_{\RM}$]
For $\eps, \delta >0$ and non-empty RM-closed $E \subset [0,1]$, there are $x_0, … ,x_k \in E$ such that $\cup_{i \leq k} B(x_i, \eps)$ has measure at least $\lambda(E)- \delta$.  
\end{princ}
%We also use this theorem with sub-scripts `$\RM$' or `$\RR$', referring to the additional representation $E$ enjoys, namely an RM-code or R2-representation respectively. 
%We note that $\WHBU^{-}_{\RM}$ is a statement in the $\L_{2}$-language.
We note that, $\WHBU^{-}_{\RM}$ is provable in $\WKL_{0}$ by \cite{jeranimo}.  With the gift of hindsight, $\WHBU_{\RM}^{-}$ also readily follows from $\HBU$, and $\ECF$ yields the result from \cite{jeranimo}.
\subsection{Littlewood's three principles}\label{kolp}
In this section, we derive $\WHBU_{\m}$ from a number of theorems that embody Littlewood's principles, like Lusin's theorem and Egorov's theorem (Section \ref{lusin}), Tao's `Littlewood-like' princples (Section \ref{datsnijdtkleinhout}), and convergence theorems (Remark \ref{conv}).  The restriction of $\WHBU$ to measurable functions is $\WHBU_{\m}$.
% and convergence theorems (Remark \ref{conv}).

\smallskip

We recall that $\WHBU^{-}_{\RM}$ from Section \ref{klopje} is provable in $\RCAo+\WKL$ by \cite{jeranimo}.
We also recall the conventions regarding the Lebesgue measure from Remark \ref{dreft}.
On a conceptual note, we will observe that the (comparative or virtual) definition of the Lebesgue measure as in Remark \ref{dreft} essentially `hard-codes' Lusin's first principle, 
yielding substantial generalisations/robust variations of the Lusin and Egorov theorems (and related results).

\subsubsection{Theorems by Lusin and Egorov}\label{lusin}
In this section, we derive $\WHBU_{\m}$ from the well-known theorems due to Lusin and Egorov. 

\smallskip

First of all, Lusin's theorem expresses that any measurable function is a continuous function on nearly all of its domain. 
This theorem constitutes the second of Littlewood's principles, and Lusin proved this theorem for real intervals in \cite{lusin1} in 1912, but it had been established previously by Borel (\cite{korrelen}), Lebesgue (\cite{korrelen2}), and Vitali (\cite{karellen}).
We note that Lusin's theorem is often\footnote{The  proofs of Lusin's theorem in e.g.\ \cite{royden1}*{p.\ 74}, \cite{folly}*{p.\ 64}, and \cite{kleinbos}*{p.\ 29}, are  basic applications of Egorov's theorem.} proved via a straightforward application of Egorov's theorem.

\smallskip

Secondly, there are multiple formulations of Lusin's theorem  (\cite{kleinbos,kesteisdenbeste, royden1,ruudgulit, wieden}) and we first study the one found in e.g.\ \cite{wieden}, as follows.  
\begin{princ}[$\LUS$]
For measurable $f:I \di \R$ and $\eps>0$, there exists measurable $E\subset [0,1]$ with measure at least $1-\eps$ and such that $f$ restricted to $E$ is continuous.  % and $\sup_{\R}|g|\leq \sup_{\R}|f|$
\end{princ}
We let $\LUS_{\RM}$ be $\LUS$ where the set $E$ is additionally assumed to be RM-closed.  
We note that Lusin's original formulation from \cite{lusin1} involves a \emph{perfect} set $E$, i.e.\ closed and no isolated points.
%We note that to approximate arbitrary measurable functions as done in second-order RM, the presence of RM-codes in Lusin's theorem seems essential. 
\begin{thm}[$\ACAo$]\label{kingmatti}~
We have $\LUS\asa \LUS_{\RM}$ and $\LUS\di \WHBU_{\m}$.
%$[\LUS+\WHBU^{-}]\di\WHBU_{\m} \di \WHBU^{-}$.  % for $\X$ equal to $\closed$, $\RR$, or $\RM$.  % we have $\WHBU_{\RR}\asa \WHBU$.
\end{thm}
\begin{proof}
For the equivalence, let $f:I \di \R$ be measurable and $\eps>0$.  Now apply $\LUS$ for $\eps/2$ to obtain measurable $E\subset [0,1]$ with measure at least $1-\eps/2$ and such that $f$ restricted to $E$ is continuous.
Since $[0,1]\setminus E$ has measure at most $\eps/2$, there is -by definition- a sequence of open intervals $(I_{n})_{n\in \N}$ such that $\cup_{n\in \N}I_{n}$ has measure at most $\eps$ and covers $[0,1]\setminus E$.   
By definition, $C:= [0,1]\setminus \cup_{n\in \N}I_{n}$ is RM-closed, has measure at least $1-\eps$, and is contained in $E$.  Thus, the set $C$ is as required for $\LUS_{\RM}$ and the equivalence follows.

\smallskip

For the second part, fix measurable $\Psi:[0,1]\di \R^{+}$ and $\eps>0$, and let $E$ be as in $\LUS_{\RM}$.  
Since $E$ has an RM-code, it is separably closed by \cite{browner}*{Theorem 3.0}.
Recall that \cite{samrep}*{Theorem 2.4} provides the Tietze extension theorem for (third-order) functions that are (epsilon-delta) continuous on a separably closed set in $[0,1]$.
Hence, there is \emph{continuous} $\Phi:[0,1]\di \R$ which equals $\Psi$ on $E$.  
Following \cite{kohlenbach4}*{\S4}, $\Phi$ has an RM-code, and applying \cite{simpson2}*{IV.2.11}, we know $\Phi$ attains its minimum on $E$, which must be non-zero. 
We therefore have $(\exists k_{0}^{0})(\forall x\in E)( |\Psi(x)|\geq \frac{1}{2^{k_{0}}})$ and applying $\WHBU^{-}_{\RM}$ for $\eps=\frac{1}{2^{k_{0}}}$ finishes the proof. 
\end{proof}
Thirdly, consider the following alternative formulation of Lusin's theorem.
\begin{princ}[$\LUS'$]
For measurable $f:I \di \R$ and $\eps>0$, there is a continuous $g:I\di \R$ such that $ \{x\in I: f(x)\ne g(x)\}$ has measure below $\eps$.  % and $\sup_{\R}|g|\leq \sup_{\R}|f|$
\end{princ}

\begin{thm}[$\ACAo$]\label{highqual}
We have $\LUS'\di\WHBU_{\m} $.
%for $\X$ equal to $\closed$, $\RR$, or $\RM$.  % we have $\WHBU_{\RR}\asa \WHBU$.
\end{thm}
\begin{proof}
For $\eps>0$ and measurable $\Psi:I\di \R^{+}$, let $g:I\di \R$ be the continuous function provided by $\LUS'$ for $\eps/2$.  
Define $E=\{ x\in [0,1]: \Psi(x)=_{\R}g(x)\}$ using $(\exists^{2})$ and note that is has measure at least $1-\eps/2$.
Since $[0,1]\setminus E$ has measure as most $\eps/2$, there is by definition a sequence $(I_{n})_{n\in \N}$ of open intervals such that the union $\cup_{n\in \N}I_{n}$ has measure at most $\eps$ and covers $[0,1]\setminus E$.
Define the RM-closed set $F$ as $[0,1]\setminus \cup_{n\in \N}I_{n}$ and note that $F\subset E$ and that $F$ has measure at least $1-\eps$.
Following \cite{kohlenbach4}*{\S4}, one readily shows that $g$ has an RM-code, and applying \cite{simpson2}*{IV.2.11}, we know $g$ attains its minimum on $F$, which must be non-zero. 
We now have $(\exists k_{0}^{0})(\forall x\in F)( |\Psi(x)|\geq \frac{1}{2^{k_{0}}})$.
Apply $\WHBU_{\RM}^{-}$ to obtain $\WHBU_{\m}$.
%Now put $N_{0}:= \lceil e_{0}\rceil$ and choose $y_{1}, \dots, y_{N_{0}+1}\in E$ as in the proof of Theorem \ref{egotripke} to obtain $\WHBU$ as in the latter.
\end{proof}
A basic fact of measure theory is that measurable functions can be expressed as the (pointwise) limit of simple functions (\cite{taomes}*{Theorem 1.3.20}).  
By converting the latter in continuous piecewise linear functions, one readily derives Lusin's theorem as in $\LUS$, i.e.\ one also obtains $\WHBU_{\m}$.

\smallskip

Fourth, we derive $\WHBU_{\m}$ from \emph{Egorov's theorem}, which was published around 1900 (\cite{egoroof, severe, korrelen2}) and expresses that a convergent sequence of measurable functions is uniformly convergent outside an arbitrarily small set (see e.g.\ \cite{taomes}*{\S1.3} and \cite{royden1}*{Ch.\ 3}).
Thus, Egorov's theorem corresponds to Littlewood's third principle and we use the formulation from \cite{wieden} as in $\EGO$ below.  In \cite{royden1}, one finds a weaker version $\EGO^{-}$, involving convergence \emph{everywhere} in the antecedent.  
%Note that the `classical' version of $\EGO$ is limited to measurable objects.  %We put $I:=[0,1]$.
\begin{princ}[$\EGO$]
Let $f_{n} : (I\times \N) \di \R$ be a sequence of measurable functions converging almost everywhere to measurable $ f :  I\di \R$, and let $\eps>0$. Then there is measurable $E\subset I$ of measure $1-\eps$ with $f_{n}$ converging uniformly to $f$ on $ E$.
\end{princ}
Let $\EGO_{\RM}^{-}$ be $\EGO^{-}$ with the extra assumption that $E$ is RM-closed.  
%Let $\IND$ be the induction axiom for all formulas of $\L_{\omega}$.  
%An obvious problem is that $\EGO$ only makes sense if the Lebesgue measure is available, which is the reason we introduced $(\blambda)$ in the previous section.  
\begin{thm}[$\ACAo$]\label{egotripke}
%The system $\ACAo+\IND+(\blambda_{0})$ proves $\EGO^{-}\di \WHBU$.  
We have $\EGO^{-}\asa \EGO_{\RM}^{-}$ and $\EGO^{-}\di \WHBU_{\m}$.  
\end{thm}
\begin{proof}
For the equivalence, proceed as in the first part of Theorem \ref{kingmatti}. 

\smallskip

For the second part, fix $\eps>0$ and measurable $\Psi:I\di \R^{+}$.  Apply $\QFAC^{1,0}$ to $(\forall x\in I)(\exists n\in \N)(\Psi(x)>\frac{1}{2^{n}}  )$ and let the resulting functional be named $\Phi$.  
Using $\exists^{2}$, one readily guarantees that $\Phi$ satisfies $x=_{\R} y \di \Phi(x)=_{0}\Phi(y)$ for $x, y\in \R$.
Now define $\Phi_{n}(x):= \min(\Phi(x), n)$.  By definition, we have that $\Phi_{n}$ converges to $\Phi$ on $I$. 
Let $E$ be as in $\EGO^{-}_{\RM}$, i.e.\ 
\be\label{uni}
(\forall \eps'>0)(\exists N\in \N)(\forall x\in E)(\forall n\geq N)(|\Phi_{n}(x)-\Phi(x)|<\eps'). 
\ee
For $\eps'=1$, take $N_{0}$ as in \eqref{uni}, and note that for $x\in  E$, we have $\Psi(x)\geq \frac{1}{2^{N_{0}}}$, by the definition of $\Phi$.
Applying $\WHBU^{-}_{\RM}$ for $\eps=\frac{1}{2^{N_{0}}}$ finishes the proof. 
\end{proof}
In conclusion, the Lusin and Egorov theorems imply $\WHBU_{\m}$ while the exact formulation of the former does not matter that much. 
In particular, the (comparative or virtual) definition of the Lebesgue measure as in Remark \ref{dreft} essentially `hard-codes' Lusin's first principle, 
allowing us to `bootstrap' e.g.\ $\LUS$ to $\LUS_{\RM}$.

\subsubsection{Littlewood-like principles}\label{datsnijdtkleinhout}
The literature contains a number of `Littlewood-like' principles, i.e.\ statements similar to Littlewood's three principles.  
We study two examples and sketch how they also imply $\WHBU_{\m}$. 

\smallskip

First of all, we study $\textsf{LLP}$, which is the (part of the) `fourth' Littlewood principle from \cite{deigo1} and one of Tao's `{Littlewood-like principles}' from \cite{taomes}*{Ex.\ 1.3.25}.  
%Note that the classical version of $\textsf{LLP}$ is restricted to measurable objects. 
\begin{princ}[$\textsf{LLP}$]
For measurable $f:\R \di \R$ and $\eps>0$, there exists a measurable set $K\subseteq I$ such that $\blambda(I\setminus K)<\eps$ and $f$ is bounded on $K$.
%there is continuous $g:\R\di \R$ such that $ \{x\in \R: f(x)\ne g(x)\}$ has Lebesgue measure less than $\eps$.  % and $\sup_{\R}|g|\leq \sup_{\R}|f|$
\end{princ}
Let $\textsf{\textup{LLP}}_{\RM}$ be $\LLP$ where $K$ is addtionally RM-closed. 
One proves the following theorem in the way as in the previous section.
\begin{thm}[$\ACAo$]
We have $\textsf{\textup{LLP}}\asa \textsf{\textup{LLP}}_{\RM}$ and $\textsf{\textup{LLP}} \di\WHBU_{\m}$.  
\end{thm}
\begin{proof}
Consider the function $\Phi$ from the proof of Theorem \ref{egotripke}. 
Then apply $\WHBU^{-}$ for the resulting set.  
\end{proof}

Secondly, we say that $f_{n}$ converges to $f$ \emph{in measure} on $X$, if for every $\eps > 0$ and $k\in \N$, there is $N\in \N$ such that for $n\geq N$, the measure of  $\{x \in  X : |f_n(x) - f(x)| \geq \eps\}$ is at most $\frac{1}{2^{k}}$.
The following theorem is called \emph{slightly weaker than Egoroff's theorem} by Royden in \cite{royden1}*{p.~72}, and connects the previous notion to pointwise convergence. 
\begin{princ}[\textsf{WTE}]
Let $f_{n},f $ be measurable and such that $(f_{n}(x))_{n\in \N}$ converges to $f(x)$ for almost all $x\in I$.  Then $(f_{n})_{n\in \N}$ converges to $ f$ in measure.
\end{princ}
%We provide a proof of the following theorem. 
\begin{thm}[$\ACAo$]
We have $\textsf{\textup{WTE}}\di \WHBU_{\m}$.  
\end{thm}
\begin{proof}
Fix $1>\eps_{0}>0$ and measurable $\Psi:I\di \R^{+}$.  Let $\Phi$ be as in the proof of Theorem \ref{egotripke} and define $\Phi_{n}(x):= \min(\Phi(x), n)$.  By definition, we have that $\Phi_{n}$ converges to $\Phi$ on $I$. 
Use $(\exists^{2})$ to define $E_{n}:=\{x \in  I : |\Phi_n(x) - \Phi(x)| \geq \eps_{0}\}$.
By \textsf{WTE}, for any $k\in \N$, there is $N\in \N$ such that the measure of $E_{n}$ is at most $1/2^{k}$ for $n\geq N_{0}$. 
For $x\in (I\setminus E_{N_{0}})$, $|\Phi_{N_{0}}(x)-\Phi(x)|<1$ implies $\Psi(x)\geq \frac{1}{2^{N_{0}+1}}$.
As for Theorem \ref{kingmatti}, $\WHBU_{\m}$ follows, and we are done.  Note that we need to use the first part of the proof of Theorem \ref{kingmatti} to obtain RM-closed sets. 
\end{proof}
%Finally, we discuss the connection between Vitali's covering theorems and convergence theorems for the Lebesgue integral.
Finally, we discuss how the well-known convergence theorems associated to the Lebesgue integral imply $\WHBU_{\m}$.
Unfortunately, establishing this result requires a non-trivial integration theory, as provided by e.g.\ the system $\ACAo+(\blambda_{0})$ from Section~\ref{onderkruipers}, \emph{and} also 
lots of technical details.  Hence, we merely sketch these results. 
\begin{rem}\label{conv}\rm
First of all, the Lebesgue integral constitutes a generalisation of the Riemann integral; one of the advantages of the former is the superior treatment of limits of integrals.  
In particular, the \emph{dominated} (resp.\ \emph{montone}) convergence theorem implies that pointwise convergence (ae) implies convergence of the associated integrals, \emph{assuming} the sequence is dominated by an integrable function (resp.\ the sequence is non-negative and monotone).  

\smallskip

Secondly, assuming the dominated convergence theorem, one derives $\WHBU_{\m}$ as in the following proof sketch.  
We implicitly use a number of properties of the Lebesgue measure and integral, provable in $\ACAo+(\blambda_{0})$ from Section~\ref{onderkruipers}, it seems.
\begin{proof} (Sketch)
Fix $\Psi:I\di \R^{+}$ and $\eps>0$ as in $\WHBU_{\m}$. 
Use $(\exists^{2})$ to define $f_{n}:I\di \R$ as $1$ if $\Psi(x)<_{\R}\frac{1}{2^{n}}$, and $0$ otherwise.  
Let $A_{n}$ be the subset of $[0,1]$ represented by $f_{n}$.  Since $(\forall n\in \N)(\forall x\in I)(f_{n}(x)\leq 1_{\R})$ and $f_{n}$ converges to $ f:= 1_{\R}$ pointwise on $I$.  We now have $\lim _{n\to \infty }\int _{I}|f_{n}-f|\,d\blambda =0$.  
By definition, we obtain $\lim_{n\di \infty} \blambda (A_{n})=1$.
Now let $N_{0}$ be such that $|1-\blambda(A_{n})|<\eps$ for $n\geq N_{0}$ and note that we may assume $\eps>\frac{1}{2^{N_{0}}}$ (just take a larger number if necessary).  
Then for $x\in I\setminus A_{N_{0}}$, we have $\Psi(x)\geq \frac{1}{2^{N_{0}}}$.
%Pick $2^{N_{0}-1}$ elements $y_{1}, \dots y_{k}$ of $A_{N_{0}}$ and note that $|I_{y_{i}}^{\Psi}|\geq \frac{1}{2^{N_{0}-1}}$ by the definition of $A_{n}$.  
As for Theorem \ref{egotripke}, $\WHBU_{\m}$ now follows. 
\end{proof}
Thirdly, assuming the monotone convergence theorem, one derives $\WHBU_{\m}$ as in the following proof sketch.  
We implicitly use a number of properties of the Lebesgue measure and integral, provable in $\ACAo+(\blambda_{0})$ from Section~\ref{onderkruipers}, it seems.
\begin{proof} (Sketch)
Fix $\Psi:I\di \R^{+}$ and $\eps>0$ as in $\WHBU_{\m}$. 
Use $(\exists^{2})$ to define $f_{n}:I\di \R$ as $1$ if $\Psi(x)<_{\R}\frac{1}{2^{n}}$, and $0$ otherwise.  
Let $A_{n}$ be the subset of $[0,1]$ represented by $f_{n}$.  Since $(\forall n\in \N)(\forall x\in I)(0\leq f_{n}(x)\leq f_{n+1}(x)\leq 1_{\R})$ and $f_{n}\di f:= 1_{\R}$ pointwise everywhere, we have $\lim _{n\to \infty }\int _{I}|f_{n}-f|\,d\blambda =0$.  
By definition, we obtain $\lim_{n\di \infty} \blambda (A_{n})=1$.
Now let $N_{0}$ be such that $|1-\blambda(A_{n})|<\eps$ for $n\geq N_{0}$ and note that we may assume $\eps>\frac{1}{2^{N_{0}}}$ (just take a larger number if necessary).  
Now proceed as in the previous proof sketch. 
\end{proof}
%We could obtain a similar result for \emph{Fatou's lemma}.  
\end{rem}
In conclusion, numerous fundamental (approximation) results from measure theory imply $\WHBU_{\m}$, sometimes over a non-trivial base theory.  
The development of measure theory even seems to go `hand in hand' with the approximation theory.

\subsection{About and around Kreuzer's measure theory}\label{onderkruipers}
\subsubsection{A measure of motivation}\label{motig}
The system of (Lebesgue) measure theory from \cite{elkhuisje} is introduced in Section~\ref{kuiken}. 
This system is $\ACA_{0}^{\omega}$ extended with the axiom $(\blambda)$ introducing an extension of the Lebesgue measure.
% We introduce a (slight) variation of 
%this system, which yields a conservative extension of $\WKL_{0}$, in Section \ref{kuiken2}.  We develop the RM of measure theory based on these systems in Section~\ref{kolp}.  
This system is \emph{non-classical}, as $(\blambda)$ implies that all subsets of the Cantor space are measurable, like in e.g.\ \cite{solo}.   
We list the main properties of Kreuzer's approach as follows. 
\begin{enumerate}
\renewcommand{\theenumi}{\roman{enumi}}
\item By \cite{elkhuisje}*{Theorem~3}, the axiom $(\blambda)$ gives rise to a $\Pi_{2}^{1}$-conservative extension of $\ACA_{0}^{\omega}$, i.e.\ the non-classical consequences of $(\blambda)$ are limited.\label{dorfu}  
\item By Theorem \ref{expla}, the Heine-Borel theorem $\HBU_{\c}$ for Cantor space is equivalent to its restriction to measurable functionals.  In this light, the assumption that all subsets of Cantor space are measurable seems innocent if we are interested in the study of compactness (and related notions).  
\item The system $\ACA_{0}^{\omega}+(\blambda)$ does not prove the Heine-Borel theorem for $2^{\N}$ (Corollary \ref{kalak}) but does prove Egorov's theorem (Theorem \ref{egotripke2}). 
\item No functional $\Theta$ as in $\SFF(\Theta)$ is computable in Kreuzer's Lebesgue measure $\blambda$ and Feferman's $\mu$ (Theorem \ref{hungy}). 
%\item There are natural open questions regarding $(\blambda)$; for instance, does the latter imply natural (non-$\Pi_{2}^{1}$ by item \eqref{dorfu}) theorems not provable in $\ACAo$?\label{kijk}
\end{enumerate}
In Section \ref{lusin}, we have derived $\WHBU_{\m}$ from Egorov's theorem.  
\subsubsection{A measure of measure theory}\label{kuiken}
We introduce the system of measure theory from \cite{elkhuisje} as in the axiom $(\blambda)$ below.

\smallskip

First of all, the system from \cite{elkhuisje} defines the Lebesgue measure on subsets of Cantor space, i.e.\ we 
need to represent such sets in $\RCAo$, as the latter is officially a type theory.  
Now, as noted in item \eqref{strijker} of Definition \ref{keepintireal}, we code sets as characteristic functions. 

\smallskip

To be absolutely clear, a set $Y\subseteq \N$ is given by a function $f_{Y}^{1}$, and we write `$n^{0}\in Y$' for $f_{Y}(n)=0$. 
There are a number of ways of representing subsets of $2^{\N}$, and we follow Kreuzer's approach from \cite{elkhuisje}*{\S2}.
%Similarly, a set `$X\subseteq \N^{\N}$' is given by a functional $F^{2}_{X}$, and we write $f^{1}\in X$ for $F_{X}(f)=0$.
Define $\sg^{1}$ as $\sg(0)=0$ and $\sg(k) = 1$ for $k > 0$. $\sg^{1 \rightarrow 1}(g)(n): = \sg(g(n))$, which we also denote as $\sg$ when there can be no confusion, maps $\N^{\N}$ to $2^{\N}$, 
and dispenses with a lot of notation.  Indeed, a set `$X\subseteq 2^{\N}$' is then given by a functional $F^{2}_{X}$, and we write `$g^{1}\in X$' in case $F_{{X}}(\sg(g))=0$, i.e.\ we quantify over Baire space but always work `modulo $\sg$', as also expressed by the last line of the axiom $(\blambda)$ just below.  For a sequence $X^{0\di \tau}$, we use $X_{n}$ to denote the $n$-th element of that sequence, as usual.    
 
\smallskip

Secondly, the system from \cite{elkhuisje} is\footnote{Note that Kreuzer's definition of $\ACAo$ in \cite{elkhuisje} is unfortunately different from ours.  In this paper, we exclusively use the definition $\ACAo\equiv \RCAo+(\mu^{2})$ from Section \ref{prelim2}.} then $\ACA_{0}^{\omega}$ extended with the axiom $(\blambda)$: % which defines the Lebesgue measure for any subset of $2^{\N}$:
\[
(\exists \pmb{\lambda}^{3})
\left(\begin{array}{l}
~~(\forall X^{2})(\blambda(X)\geq 0)\wedge (\forall f\in C)(X(f)\ne 0)\di (\blambda(X)=0)\\
\wedge (\forall X^{0\di 2})\big(\blambda(\cup_{i\in \N}X_{i}')=_{\R}\sum_{i=0}^{\infty}\blambda(X_{i}')\big) \\
\wedge (\forall s^{0^{*}})\big(\blambda([s])=_{\R}2^{-|s|}\big) \\
\wedge (\forall X^{2})\big( \blambda(X)=_{\R}\blambda(\lambda f^{1}.\sg(X(f)))=_{\R}  \blambda(X(\lambda n^{0}.\sg(f(n)) )) \big)
\end{array}\right),
\]
where $X_{i}' := X_{i} \setminus \cup_{j<i} X_{j}'$.
%[In the sequel, we will write $X_{i}'$, $U_{i}'$ etc. with this interpretation without further explanation.]
 The last line of $(\blambda)$ indicates that $\blambda$ is compatible with our coding of subsets of $2^{\N}$.
By \cite{elkhuisje}*{Theorem~3}, $\ACA_{0}^{\omega}+(\blambda)$ is a $\Pi_{2}^{1}$-conservative extension of $\ACA_{0}^{\omega}$.  Similar to \cite{elkhuisje}, we will make use of a Skolem constant $\blambda$ added to the language.  % of $\RCAo$. 
As is well-known, $2^{\N}$ and $[0,1]$ are measure-theoretically equivalent.  Indeed, $\r(\alpha):=\sum_{n=0}^{\infty}\frac{\alpha(i)}{2^{n+1}}$ is a measure-preserving surjection, and $(\blambda)$ thus defines a measure on $[0,1]$.
In a strong system like $\ZFC$, we can prove that this measure  will be identical to the Lebesgue measure on the class of Lebesgue-measurable sets.

\smallskip

Thirdly, $(\blambda)$ is `non-classical' in nature as it implies that \emph{all} subsets of $2^{\N}$ are measurable.  % (though translation-invariance is not mentioned).  
Indeed, the Axiom of Choice \textsf{AC} is known to yield non-measurable sets (see e.g.\ \cite{taomes}*{\S1.2.3}) if we require the measure to be countably additive \emph{and translation-invariant}.  
Due to the absence of the latter requirement, $(\blambda)$ taken as a statement in $\ZFC$ does not violate \textsf{AC}, but it implies large cardinal\footnote{The axiom $(\blambda)$ implies there is a weakly inaccessible cardinal below the continuum, namely a weakly Mahlo cardinal. In particular, the assumption violates the axiom \textsf{V = L}.} axioms.  

\subsubsection{Some results}
We establish some results in and about Kreuzer's framework.

\smallskip

First of all, whether or not all subsets of $2^{\N}$ are measurable turns out not to have an influence on open-cover compactness. Indeed, by Theorem~\ref{expla}, Heine-Borel compactness as in $\HBU_{\c}$ does not really change if we restrict to measurable functionals.
In this light, assuming that all subsets of $2^{\N}$ are measurable does not really change the strength of $\HBU_{\c}$,  i.e.\ one should feel free to use this assumption when studying compactness, like in the form of $(\blambda)$.  
%Finally, we prove two theorems which provide a justification for $(\blambda)$ (and its below variations) in the study of compactness formulated as in $\HBU_{\c}$. 
\begin{thm}\label{expla}
Given $(\exists^2)$, the theorem $\HBU_\c$ follows from the restriction of $\HBU_\c$ to measurable functionals.  %, and thus t follows that we cannot prove HBU_c for measurable functions from Pi^1_k-CA_0^omega.
\end{thm}
\begin{proof}
Fix $f,g \in C$ and define $h = \langle f,g\rangle $ by $ h(2n) = f(n)$ and $h(2n+1) = g(n)$.
Let $C^0$ be the set of $\langle \lambda x^{0}.0,f\rangle $ such that $f\in  C$. Then $C^0$ is a compact subset of $C$ of measure $0$, so every total function that is continuous outside $C^0$ will be measurable. Moreover $C^0$ is homeomorphic to $C$.
Let $F:C \rightarrow \N$ be arbitrary. Define $F_{0}(\langle \lambda x.0,f\rangle ) = 2F(f) + 1$, and define  $F_{0}(h) = n$ for the least $n$ such that $C_{\bar h(n)}$ is disjoint from $C^0$ if $h$ is not in $C^0$.
Then $F_{0}$ is measurable, and if we apply $\HBU_\c$ to $F_{0}$, $C^0$ can only be covered by $C_{\bar h(F^0(h))}$ for $h \in C^0$, and by the homeomorphism, we obtain a finite sub-cover of the cover of $C$ induced by $F$.
\end{proof}
In hindsight, the previous theorem is not that surprising: the Axiom of Choice is not needed to prove $\HBU_{\c}$ (see \cite{dagsamV}*{\S4.1}): the latter is provable in $\Z_{2}^{\Omega}$.  
Since the existence of non-measurable sets is intimately connected to the Axiom of Choice, it stands to reason the latter has no influence on $\HBU_{\c}$.  Since special fan functionals compute realisers for $\HBU_{\c}$, 
we expect the following complimentary result, where $\textsf{LMC}(\blambda)$ is $(\blambda)$ without the leading existential quantifier.  
\begin{thm}\label{hungy}
No functional $\Theta$ as in $\SFF(\Theta)$ is computable in any $\blambda$ as in $\textup{\textsf{LMC}}(\blambda)$ and Feferman's $\mu$. 
\end{thm}
\begin{proof}
There is a partial functional of type $1 \di 1$ computable in $\mu$ which to a code for a Borel-subset $B$ of $C$ computes a binary representation of $\blambda(B)$, see \cite{Sacks.high}*{Section IV.1} or \cite{dagsamII}*{Prop.\ 3.23}. 
If $\Theta$ were computable in $\blambda$, we could use this and the recursion theorem (for S1-S9) to show that whenever we have an index $e$ for computing a functional $F$ from $\mu$ and some $f^1$, we can find a value of $\Theta(F)$ computable uniformly from $e$ and $f$. 
Since there is arithmetical $F$ such that $\Theta(F)$ cannot be hyperarithmetical (see \cite{dagsam, dagsamII}), this is impossible. Hence, no $\Theta$ is computable in $\blambda$ and $\mu$.
%
%It will be much harder to  prove that Lambda is not computable in lambda, if possible at all.
\end{proof}
Secondly, we now show that \emph{Heine-Borel} compactness as in $\HBU_{\c}$ is not provable from in Kreuzer's framework. 
\begin{cor}\label{kalak}
The system $\ACAo+(\blambda)$ cannot prove $\HBU_{\c}$. 
\end{cor}
\begin{proof}
Similar to the proof of the theorem, there is arithmetical $F_{0}$ such that for any finite sub-cover $\cup_{i\leq k}[\overline{f_{i}}F_{0}(f_{i})]$ of the canonical cover $\cup_{f\in C}[\overline{f}F_{0}(f)]$, the finite sequence $\langle f_{1}, \dots, f_{k}\rangle$ is not hyperarithmetical (see \cite{dagsam, dagsamII}).  
As in \cite{elkhuisje}, denote by $T_{0}$ the sub-system of G\"odel's system $T$, where primitive recursion is restricted to the recursor $\mathbf{R}_{0}$, 
and let $T_{0}[F]$ be $T_{0}$ extended with the function(al) $F$.  
Then $T_{0}[F]$ satisfies $(\blambda)$ by \cite{elkhuisje}*{Lemma~7}.  
Hence, there is a model of $\ACAo+(\blambda)$ in which $\HBU_{\c}$ is false, as it contains $F_{0}$ but the finite sub-cover for the associated canonical cover is lacking.   

\smallskip

Alternatively, the proof of \cite{elkhuisje}*{Theorem 3} establishes the following term extraction procedure: if for arithmetical $A$, $\ACAo+(\blambda)$ proves $(\forall f^{1})(\exists g^{1})A(f, g)$, then a term $t$ can be extracted from this proof such that $\ACAo$ proves $(\forall f^{1})A(f, t(f, \mu))$.  Note the essential role of $\mu^{2}$ in the conclusion. 
If $\ACAo+(\blambda)$ proves $\HBU_{\c}$, we consider the latter restricted to the arithmetical functional $F_{0}$ from the previous paragraph.  Note that we may replace\footnote{In the case of the unit interval, this means that if we have a finite sub-cover for $[0,1]\cap \Q$, we also have a finite sub-cover for $[0,1]$ by adding the end-points of the intervals of the first covering. 
} the innermost universal  quantifier over $C$ by a numerical quantifier.  Hence, the resulting sentence (modulo some applications of $\mu^{2}$) has the right format for applying the previous term extraction result.  
However, this means we obtain a hyperarithmetical finite sub-cover for the canonical cover corresponding to $F_{0}$, a contradiction.   
%The system $\ACAo+(\blambda)$ is $\Pi_{2}^{1}$-conservative over $\ACA_{0}$, and $\ATR_{0}$ is not by \cite{simpson2}*{VIII.4.20}. 
\end{proof}
By \cite{elkhuisje}*{Remark 13}, $\FIVE^{\omega}+(\blambda)$ similarly cannot prove $\HBU_{\c}$, and the same for stronger systems.  
A lot of details need to be worked out to establish this result, however.   The crucial part of the previous theorem is that 
$\HBU_{\c}$ restricted to \emph{arithmetically defined} covers already yields non-hyperarithmetical functions (and in fact $\ATR_{0}$; see \cite{dagsam,dagsamII}).  
In Section \ref{CTM}, we do establish that $\FIVE^{\omega}+\WHBU$ cannot prove $\HBU_{\c}$. % where the former is the Vitali covering theorem.
%This discovery was made in response to a question of Steve Simpson on the strength of $\HBU_{\c}$ \emph{restricted to Borel sets}.  

\smallskip

Finally, {Egorov's theorem} can be established in Kreuzer's framework.  
%This theorem was first published around 1900 (\cite{egoroof, severe, korrelen2}) and expresses that a convergent sequence of measurable functions is uniformly convergent outside an arbitrarily small set (see e.g.\ \cite{taomes}*{\S1.3} and \cite{royden1}*{Ch.\ 3}).  In other words, Egorov's theorem corresponds to Littlewood's third principle.
\begin{thm}\label{egotripke2}
The system $\ACAo+(\blambda)$ proves Egorov's theorem as follows:
\begin{center}
let $f_{n} : (I\times\N) \di \R$ be a sequence converging almost everywhere to $ f :  I\di \R$, and let $\eps>0$. Then there is $E\subset I$ such that $\blambda(E)>1-\eps$  and $f_{n}$ converges uniformly to $f$ on $ E$.
\end{center}  
\end{thm}
\begin{proof}
%Like totally!  What about $\ACAo+(\blambda_{0})+\WHBU$?
The `usual' proof (see e.g.\ \cite{royden1}*{Ch.\ 3.6}) goes through as follows.  First of all, the Lebesgue measure as in $(\blambda)$ is `continuous from above', i.e.\ we have
\be\label{ergotoko}
[(\forall n\in \N)(E\subseteq E_{n+1}\subseteq E_{n}\subset [0,1])\wedge E=\cap_{n\in \N}E_{n} ]\di  \blambda(E)=\lim_{n\di \infty} \blambda(E_{n}).  
\ee
Indeed, the proof of \eqref{ergotoko} in e.g.\ \cite{royden1}*{p.\ 63} amounts to nothing more than defining a \emph{disjoint} collection from the $E_{k}$'s (which can be done using $\mu^{2}$), and then applying 
countable additivity, included in $(\blambda)$, to obtain the consequent of \eqref{ergotoko}.  

\smallskip

Secondly, define $E_{{n,k}}:=\cup _{{m\geq n}}\{x\in I : |f_{m}(x)-f(x)|\geq {\frac  1k}\}$ and note that $\blambda(\cap_{n\in \N}E_{n, k})=0$ by the assumption that $f_{n}\di f$ almost everywhere. 
Applying \eqref{ergotoko}, we obtain $\lim_{n\di \infty} \blambda(E_{n,k})=0$, which implies the following:
\be\label{convje}\textstyle
(\forall k\in \N)(\exists N\in \N)(\forall n\geq N)(\blambda(E_{n,k})<\frac{1}{2^{k+1}}).
\ee
Fix $\eps>0$ and use $(\mu^{2})$ in \eqref{convje} to find $g^{1}$ such that $(\forall k\in \N)(\blambda(E_{g(k),k})<\frac{\eps}{2^{k+1}})$ and define $E:=\cup_{k\in \N}E_{g(k),k}$.
It is now straightforward to show that the set $E$ is as required for $\EGO$.  Indeed, countable additivity implies that $\blambda(E)\leq \sum_{n=0}^{\infty}\blambda(E_{g(k), k})\leq \eps$ while for $x\in I\setminus E$, the rate of uniform convergence is $g$.  % given by $g$, by definition, and we are done.
\end{proof}
In conclusion, while Kreuzer's framework can establish fundamental results in measure theory like Egorov's theorem, we cannot hope to prove any theorem based on the (uncountable) Heine-Borel theorem by Corollary \ref{kalak}.
As shown in \cite{dagsamIII}, many basic results in third-order arithmetic imply $\HBU$, including the development of the \emph{gauge integral} (see Section \ref{bintro}).  However, 
the latter for absolutely integrable functions is exactly the Lebesgue integral (\cite{bartle1}).  

\subsection{Alternative approaches}\label{GL}
The above results suggest that $\WHBU$ is essential to the development of the measure theory, the Lebesgue measure and integral in particular.  
To assuage any fears that these results depend on our choice of framework, we now consider a \emph{very} different framework for the 
Lebesgue integral, and show that $\WHBU$ is essential there too. 

\smallskip

One alternative framework is the \emph{gauge integral} (\cite{zwette,mullingitover}) restricted to \emph{bounded} functions (Sections \ref{bintro} and \ref{grm}).  
Basic properties of the (general) gauge integral were shown in \cite{dagsamIII}*{\S3.3} to be equivalent to $\HBU$.     
We discuss further applications of our alternative approach in Remark \ref{topent}, including \emph{topological entropy} 

\smallskip

Finally, to drive home the point that $\WHBU$ emerges everywhere in integration theory, we establish the following result in Section \ref{netz}: 
the monotone convergence theorem \emph{for nets} of functions and the Riemann integral implies $\WHBU$.

\subsubsection{Restricting the gauge integral}\label{bintro}
The \emph{gauge integral} is a generalisation of the Lebesgue and improper Riemann integral; it was introduced by Denjoy (in a different from) around 1912 and developed further by Lusin, Perron, Henstock, and Kurzweil (\cite{zwette}).  
The definition of the gauge integral in Definition \ref{GID} is highly similar to the Riemann integral (and simpler than Lebesgue's integral), but boasts a maximal `closure under improper integrals', known as \emph{Hake's theorem} (\cite{bartle1}*{p.\ 195}). 

\smallskip

The aforementioned scope and versatility of the gauge integral comes at a non-trivial `logical' cost: as established in \cite{dagsamIII}*{\S3}, $\HBU$ is equivalent to many basic properties of the gauge integral, including uniqueness.  The additivity of the gauge integral also requires \emph{discontinuous} functions on $\R$, and the resulting system is at the level of $\ATR_{0}$ by \cite{dagsam}*{Cor.\ 6.7} and \cite{dagsamIII}*{Theorem 3.3}.  It is then a natural question if for natural \emph{sub-classes} of functions, a weaker system, e.g.\ at the level of $\ACA_{0}$, suffices to develop the associated \emph{restricted} gauge integral.   

\smallskip

The positive answer to this question starts with a fundamental result, namely that for \emph{bounded} $f$ on bounded intervals, the following are equivalent: $f$ is measurable, $f$ is gauge integrable, and $f$ is Lebesgue integrable (\cite{bartle1}*{p.\ 94}).  Thus, the bounded functions on $[0,1]$ constitute a sub-class with natural properties.  
Furthermore, the Riemann sum of bounded functions is `well-behaved':  the former sum does not vary much\footnote{Using the notions defined in Definition \ref{GID}, if $(\forall x\in I)(|f(x)|\leq d)$ for $d\in \N$, then $S(f, P)$ only varies $\eps /d$ if we change $f$ in an interval $(a,b)\subset I$ such that $|b-a|<\eps$.} if we change the function on a small sub-interval.  
Hence, we may weaken $\HBU$ to only apply to `most' of $I$, which is exactly $\WHBU$: the latter expresses that we have a finite sub-cover of any canonical cover, for `most' of $I$, i.e.\ a subset of measure $1-\eps$ for any $\eps>0$.     

\smallskip

The previous discussion leads to the following definition. 
For brevity, we assume bounded functions on $I$ to be bounded by $1$.  
The crucial and (to the best of our knowledge) new concepts are `$\eps$-$\delta$-fine' and the $L$-integral in items \eqref{hurfo1} and \eqref{lakel22}.
All other notions are part of the (standard) gauge integral literature (see e.g. \cite{mullingitover}).  
\bdefi\label{GI}[Integrals]\label{GID}~   
\begin{enumerate}
\renewcommand{\theenumi}{\roman{enumi}}
\item A gauge on $I\equiv[0,1]$ is any function $\delta:\R\di \R^{+}$.  
\item A sequence $P:=(t_{0}, I_{0}, \dots, t_{k}, I_{k})$ is a \emph{tagged partition} of $I$, written `$P\in \textsf{tp}$', if the `tag' $t_{i}\in \R$ is in the interval $ I_{i}$ for $i\leq k$, and the $I_{i}$ partition $I$.
\item If $\delta$ is a gauge on $I$ and $P=(t_{i}, I_{i})_{i\leq k}$ is a tagged partition of $I$, then $P$ is \emph{$\delta$-fine} if $I_{i}\subseteq [t_{i}-\delta(t_{i}), t_{i}+\delta(t_{i})]$ for $i\leq k$.\label{hurfo}
\item If $\delta$ is a gauge on $I$ and $P=(t_{i}, I_{i})_{i\leq k}$ is a tagged partition of $I$ and $\eps>0$, then $P$ is \emph{$\eps$-$\delta$-fine} if $\cup_{i=0}^{k}\widetilde{I}_{i}$ has measure at least $ 1-\eps$, where 
$\widetilde{I}_{i}$ is ${I}_{i}$ if $I_{i}\subseteq [t_{i}-\delta(t_{i}), t_{i}+\delta(t_{i})]$, and empty otherwise.\label{hurfo1}
\item For a tagged partition $P=(t_{i}, I_{i})_{i\leq k}$ of $I$ and any $f$, the \emph{Riemann sum} $S(f, P)$ is $\sum_{i=0}^{n}f(t_{i})|I_{i}|$, while the \emph{mesh} $\|P\|$ is $\max_{i\leq n}|I_{i}|$.
\item A function $f:I\di \R$ is \emph{Riemann integrable} on $I$ if there is $A\in \R$ such that $(\forall \eps>_{\R}0)(\exists\delta>_{\R}0)(\forall P\in \textsf{tp})(\|P\|\leq_{\R} \delta\di |S(f, P)-A|<_{\R}\eps)$.\label{lakel}
\item A function $f:I\di \R$ is \emph{gauge integrable} on $I$ if there is $A\in \R$ such that
$(\forall \eps>_{\R}0)(\exists\delta:\R\di \R^{+})(\forall P\in \textsf{tp})(\textup{$P$ is $\delta$-fine }\di |S(f, P)-A|<_{\R}\eps)$.\label{lakel2}
\item A \textbf{bounded} $f:I\di \R$ is \emph{L-integrable} on $I$ if there is $A\in \R$ such that
$(\forall \eps>_{\R}0)(\exists\delta:\R\di \R^{+})(\forall P\in \textsf{tp})(\textup{$P$ is $\eps$-$\delta$-fine }\di |S(f, P)-A|<_{\R}\eps)$.\label{lakel22}
\item A \emph{gauge modulus} for $f$ is a function $\Phi:\R\di (\R\di \R^{+})$ such that $\Phi(\eps)$ is a gauge as in items \eqref{lakel2} and \eqref{lakel22} for all $\eps>_{\R}0$.  
\end{enumerate}
\edefi
The real $A$ from items \eqref{lakel} and \eqref{lakel2} in Definition \ref{GID} is resp.\ called the Riemann and gauge integral. 
We always interpret $\int_{a}^{b}f $ as a gauge integral, unless explicitly stated.  We abbreviate `Riemann integration' to `R-integration', and the same for related notions.  
The real $A$ in item \eqref{lakel22} is called the Lebesgue (or L-) integral or \emph{restricted gauge integral} due to the extra condition that $f$ be bounded on $I$.   

\smallskip

Finally, using the Axiom of Choice, a gauge integrable function always has a gauge modulus, but this is not the case in weak systems like $\RCAo$.
However, to establish the \emph{Cauchy criterion} for gauge integrals as in \cite{dagsamIII}*{\S3.3}, a gauge modulus is essential.  
For this reason, we sometimes assume a gauge modulus when studying the RM of the gauge integral in Section \ref{grm}.  
Similar `constructive enrichments' exist in second-order RM, as established by Kohlenbach in \cite{kohlenbach4}*{\S4}.

\subsubsection{Reverse Mathematics of the restricted gauge integral}\label{grm}
We show that basic properties of the L-integral imply (or are equivalent to) $\WHBU$ as follows.   
We have based this development on Bartle's introductory monograph \cite{bartle1} and \cite{dagsamIII}*{\S3.3}. 

\smallskip

First of all, we show that $\WHBU$ is equivalent to the uniqueness of the L-integral, and to the fact that the latter extends the R-integral.
Note that the names of the two items in the theorem are from \cite{bartle1}*{p.\ 13-14}.  Also note that a Riemann integrable function is bounded, even in $\RCAo$.
\begin{thm}\label{firstje} Over $\ACA_{0}^{\omega}$, the following are equivalent to $\WHBU$:
\begin{enumerate}
\renewcommand{\theenumi}{\roman{enumi}}
\item Uniqueness: If a bounded function is L-integrable on $[0,1]$, then the L-integral is unique. \label{itemone}
\item Consistency: If a function is R-integrable on $[0,1]$, then it is L-integrable there, and the two
integrals are equal.\label{itemtwo}
\end{enumerate}
\end{thm}
\begin{proof}
We prove $\WHBU\di \eqref{itemone}\di \eqref{itemtwo}\di \WHBU$.  % where only the first implication requires $(\mu^{2})$.
To prove that $\WHBU$ implies \emph{Uniqueness}, assume the former, let $f$ be bounded and gauge integrable on $I$ and suppose $f$ satisfies for $i=1,2$ (where $A_{i}\in \R$) that:
\be\label{loker}
(\forall \eps>0)(\exists\delta_{i}:\R\di \R^{+})(\forall P\in \textsf{tp})(\textup{{$P$ is $\eps\textup{-}\delta_{i}$-fine} }\di |S(f, P)-A_{i}|<\eps).
\ee
Fix $\eps>0$ and the associated $\delta_{i}:\R\di \R^{+}$ in \eqref{loker} for $i=1,2$.  We define the gauge $\delta_{3}:\R\di \R^{+}$ as $\delta_{3}(x):=\min (\delta_{1}(x), \delta_{2}(x))$.  
By definition, a partition which is $\eps$-$\delta_{3}$-fine, is also $\eps$-$\delta_{i}$-finite for $i=1,2$.  
Now assume there is $P_{0}\in \textsf{tp}$ which is $\eps$-$\delta_{3}$-fine, and note that we obtain the following by applying \eqref{loker}:
\[
|A_{1}-A_{2}|=_{\R}|A_{1}-S(f, P_{0})+S(f, P_{0})-A_{2}|\leq_{\R}|A_{1}-S(f, P_{0})|+|S(f, P_{0})-A_{2}|\leq_{\R} 2\eps.
\]
Hence, we must have $A_{1}=_{\R}A_{2}$, and \emph{Uniqueness} follows.  What remains is to prove that for every gauge $\delta$ there \emph{exists} a $\eps$-$\delta$-fine tagged partition.
We emphasise the crucial nature of this existence: \eqref{loker} is vacuously true if there is no $\eps$-$\delta_{i}$-fine tagged partition; in other words: we can \emph{only} make meaningful use of the conclusion 
of \eqref{loker}, \emph{if} we show the existence of a $\eps$-$\delta_{i}$-fine tagged partition.  

\smallskip

Thus, fix $\delta:\R\di \R^{+}$ and apply $\WHBU$ to $\cup_{x\in I}(x-\delta(x), x+\delta(x))$ to obtain $w=\langle y_{0}, \dots, y_{k}\rangle$ in $I$ such that the measure of $\cup_{n=0}^{k}I_{y_{n}}^{\delta}$ is at least $ 1-\eps$.   This finite sequence is readily converted into a tagged partition $P_{0}:=(z_{j}, I_{j})_{j\leq l}$ (with $l\leq k$ and $z_{j}\in w$ for $j\leq l$) by removing overlapping segments and omitting redundant intervals `from left to right'.  By definition, $z_{j}\in I_{j}\subset (z_{j}-\delta(z_{j}), z_{j}+\delta(z_{j}))$ for $j\leq l$, i.e.\ $P_{0}$ is $\eps$-$\delta$-fine.  While the previous two steps are straightforward, it should be noted that (i) $\WHBU$ is essential by the equivalences in the theorem, and (ii) to convert $w$ into a tagged partition, we need to compare real numbers (in the sense of deciding whether $x>_{\R}0$ or not) and this operation is only available in $\ACAo$.

\smallskip

To prove that \emph{Uniqueness} implies \emph{Consistency}, note that `$P$ is $\eps$-$d_{\delta}$-fine' follows from `$\|P\|\leq \delta$' for the gauge $d_{\delta}:\R\di \R^{+}$ which is constant $\delta>0$, and any $\eps>0$.
Rewriting the definition of Riemann integration with the first condition, we observe that an R-integrable function $f$ is also L-integrable (with a constant gauge $d_{\delta}$ for every choice of $\eps>0$).    
The assumption \emph{Uniqueness} then guarantees that $A$ is the only possible L-integral for $f$ on $I$, i.e.\ the two integrals are equal.      

\smallskip

To prove that \emph{Consistency} implies $\WHBU$, suppose the latter is false, i.e.\ there is $\Psi_{0}:\R\di \R^{+}$ and $\eps_{0}>0$ such that for all $y_{1}, \dots, y_{k} \in I$, the measure of $\cup_{n=0}^{k}I_{y_{n}}^{\Psi(y_{n})}$ is below $ 1-\eps_{0}$.  
Note that the same property holds for all $\eps\leq \eps_{0}$.  
Now let $f:I\di \R$ be R-integrable with R-integral $A\in \R$.
Define the gauge $\delta_{0}$ as $\delta_{0}(x):={\Psi_{0}(x)}$ and note that for any $P\in \textsf{tp}$ and $\eps\leq \eps_{0}$, we have that $P$ is \emph{not} $\eps$-$\delta_{0}$-fine, as the tags of $P$ would otherwise provided the reals $y_{i}$ from $\WHBU$.  % (provided by the tags of $P$).  
Hence, \eqref{corefu} below  is vacuously true, as the underlined part is false:
\be\label{corefu}
(\forall \eps\in (0, \eps_{0}])(\forall P\in \textsf{tp})(\underline{\textup{$P$ is $\eps$-$\delta_{0}$-fine }}\di |S(f, P)-(A+1)|<\eps).
\ee
However, \eqref{corefu} implies that $f$ is L-integrable with L-integral $A+1$, i.e.\ \emph{Consistency} is false as the R and L-integrals of $f$ differ.  %Note that $\delta_{0}$ also provides a gauge \emph{modulus} by \eqref{corefu} in case $\neg\HBU$.  
\end{proof}
The previous proof is similar to the related equivalence for $\HBU$ and \emph{uniqueness} and \emph{consistency} for the (unrestricted) gauge integral from \cite{dagsamIII}*{\S3.3}.  Other results in the latter section can be developed along the 
same lines with similar proofs.  For this reason, we only mention these results without proof.    
\begin{thm}\label{coromag} Over $\ACA_{0}^{\omega}+\QFAC^{2,1}$, the following are equivalent to $\WHBU$:
\begin{enumerate}
\renewcommand{\theenumi}{\roman{enumi}}
\item There exists a bounded function which is not $L$-integrable with a modulus.\label{itemwon}
\item \(Hake\) If a bounded function $f$ is $L$-integrable on $I$ with modulus and R-integrable on $[x,1]$ for $x>0$, the limit of R-integrals $\lim_{x\di 0+}\int_{x}^{1}f$ is $\int_{0}^{1}f$.\label{haken}
\item \(weak Hake\) If a bounded function $f$ is $L$-integrable on $I$ with modulus and R-integrable on $[x,1]$ for $x>0$, the limit of R-integrals $\lim_{x\di 0+}\int_{x}^{1}f $ exists.\label{zwaken}
\end{enumerate}
\end{thm}
We point out that the function $\kappa:I\di \R$ from \cite{dagsamIII}*{\S3.3} is \emph{unbounded}, i.e.\ the previous theorems do not apply.  
This function $\kappa$ is used to show that $\HBU$ is equivalent to the existence of a gauge integrable function that is not Lebesgue integrable, i.e. for which the absolute value is not gauge integrable.  

\smallskip

Finally, we discuss other possible applications of $\WHBU$.  We emphasise the speculative nature of the following remark.   
\begin{rem}[Topological entropy]\label{topent}\rm
The notion of \emph{topological entropy} $h(\varphi)$ is introduced in \cite{adelheid} for a continuous function $\varphi:X\di X$ and compact space $X$.  
The number $h(\varphi)$ is non-negative (possibly $+\infty$ by \cite{adelheid}*{Ex.\ 3}) and crucially depends on open-cover compactness as follows: 
$h(\varphi)$ is the supremum of $h(\varphi, \mathfrak{A})$ over all open covers $\mathfrak{A}$ of $X$.  In turn, $h(\varphi, \mathfrak{A})$ is $\lim_{n\di +\infty}\frac{1}{n}H(\mathfrak{A}\vee \varphi^{-1}\mathfrak{A}\vee \dots \vee \varphi^{-n+1}\mathfrak{A})$, 
where $\mathfrak{A}\vee \mathfrak{B}=\{A\cap B: A\in \mathfrak{A}\wedge B\in \mathfrak{B}\}$.  Finally, $H(\mathfrak{A})=\log(N(\mathfrak{A}))$ is the \emph{entropy} of the cover $\mathfrak{A}$, where $N(\mathfrak{A})$ is the minimum number of sets in $\mathfrak{A}$ that still cover $X$.  The similar notion of \emph{metric} entropy is based on \emph{partitions and distance} rather than the size of sub-covers.  Hence, basic properties 
of metric entropy can be established in relatively weak systems (compared to say the hardness of $\HBU$).     

\smallskip

Moreover, it is not a leap of the imagination that basic properties of $h(\varphi)$ imply $\HBU$, even if $X=[0,1]$.  The same holds for the \emph{variational principle} that connects topological entropy to metric entropy (see e.g.\ \cite{goeiewijn}).    
To avoid the use of $\HBU$, and the associated `explosion'\footnote{By \cite{dagsam}*{\S6}, the combination of $\HBU$ and $(\exists^{2})$ implies $\ATR_{0}$ over $\RCAo+\QFAC^{2,1}$.}, one works with $\WHBU$ instead as follows: one defines $H_{0}(\mathfrak{A},\eps)=\log(N_{0}(\mathfrak{A}, \eps))$ where $N_{0}(\mathfrak{A}, \eps)$ is 
the minimum number of sets in $\mathfrak{A}$ such that the union has measure at least $1-\eps$.  
We then put $h_{0}(\varphi, \mathfrak{A}, \eps):=\lim_{n\di +\infty}\frac{1}{n}H_{0}(\mathfrak{A}\vee \varphi^{-1}\mathfrak{A}\vee \dots \vee \varphi^{-n+1}\mathfrak{A},\eps)$ and $h_{0}(\varphi,\eps)$ is the supremum over covers $\mathfrak{A}$ of $X$.  Assuming the supremum (involving the metric entropy) from the aforementioned variational principle is finite, $h_{0}(\varphi,\frac{1}{2^{n}})$ is a bounded increasing sequence, and hence $\lim_{n\di +\infty} h_{0}(\varphi, \frac{1}{2^{n}})$ exists.  This limit seems a `worthy' stand-in for $h(\varphi)$ when the latter is not well-defined (due to the absence of $\HBU$).
% and $\eps>0$.   
\end{rem}

\subsubsection{Nets and the Riemann integral}\label{netz}
Lest there be any doubt that $\WHBU$ is to be found everywhere in integration theory, 
we show in this section that the monotone convergence theorem for \emph{nets} and the Riemann integral implies $\WHBU$.

\smallskip

First of all, the notion of \emph{net} is the generalisation of the concept of \emph{sequence} to (possibly) uncountable index sets and any topological space.   
Nets were introduced about a century ago by Moore-Smith (\cite{moorsmidje}), 
who also proved e.g.\ the Bolzano-Weierstrass, Dini and Arzel\`a theorems \emph{for nets}. 
The RM-study of these theorems may be found in \cite{samcie19, samwollic19, samnetspilot}, and each of them implies $\HBU$.
Moreover, only nets indexed by subsets of Baire space are used for these results, i.e.\ a `step up' from sequences gives rise to $\HBU$, and the same for this paper by Definition \ref{strijker2}.    
\bdefi[Nets]\label{nets}
A set $D\ne \emptyset$ with a binary relation `$\preceq$' is \emph{directed} if
\begin{enumerate}
 \renewcommand{\theenumi}{\alph{enumi}}
\item The relation $\preceq$ is transitive, i.e.\ $(\forall x, y, z\in D)([x\preceq y\wedge y\preceq z] \di x\preceq z )$.
\item The relation $\preceq$ is reflexive, i.e.\ $(\forall x\in D)(x\preceq x)$.
\item For $x, y \in D$, there is $z\in D$ such that $x\preceq z\wedge y\preceq z$.\label{bulk}
\end{enumerate}
For such $(D, \preceq)$ and topological space $X$, any mapping $x:D\di X$ is a \emph{net} in $X$.  
\edefi
Since nets are the generalisation of sequences, we write $x_{d}$ for $x(d)$ to emphasise this connection.  
The relation `$\preceq$' is often not explicitly mentioned; we write `$d_{1},\dots, d_{k}\succeq d$' for $(\forall i\leq k)(d_{i}\succeq d)$.  
We shall only consider nets indexed by subsets of $\N^{\N}$, as follows. 
\bdefi[Directed sets and nets in $\RCAo$]\label{strijker2}
A `subset $D$ of $\N^{\N}$' is given by its characteristic function $F_{D}^{2}\leq_{2}1$, i.e.\ we write `$f\in D$' for $ F_{D}(f)=1$ for any $f\in \N^{\N}$.
A `binary relation $\preceq$ on a subset $D$ of $\N^{\N}$' is given by the associated characteristic function $G_{\preceq}^{(1\times 1)\di 0}$, i.e.\ we write `$f\preceq g$' for $G_{\preceq}(f, g)=1$ and any $f, g\in D$.
Assuming extensionality on the reals as in item \eqref{rext} of Definition \ref{keepintireal}, we obtain characteristic functions that represent subsets of $\R$ and relations thereon.  
Using pairing functions, it is clear we can also represent sets of finite sequences (of real numbers), and relations thereon.  
\edefi
Thus, a net $x_{d}:D\di \R$ in $\RCAo$ is nothing more than a type $1\di 1$-functional with extra structure on its domain $D\subseteq \N^{\N}$ provided by $\preceq$.
The definitions of convergence and increasing net are as follows, and now make sense in $\RCAo$.
\bdefi[Convergence of nets]\label{convnet}
If $x_{d}$ is a net in $X$, we say that $x_{d}$ \emph{converges} to the limit $\lim_{d} x_{d}=y\in X$ if for every neighbourhood $U$ of $y$, there is $d_{0}\in D$ such that for all $e\succeq d_{0}$, $x_{e}\in U$. 
%Similarly, a point $y\in X$ is a \emph{limit point} of $x_{d}$ if for every neighbourhood, there is $d\in D$ such that $x_{d}\in U$.
\edefi
It goes without saying that for nets \emph{of functions} $f_{d}:(D\times [0,1])\di \R$, properties of $f_{d}(x)$ like continuity pertain to the variable $x$, while the net is indexed by $d\in D$. 
\bdefi[Increasing net]\label{inet}
A net $f_{d}:(D\times I)\di \R$ is \emph{increasing} if $a\preceq b$ implies $f_{a}(x)\leq_{\R} f_{b}(x) $ for all $x\in I$ and $a,b\in D$.
\edefi
We formulate the monotone convergence theorem $\MCT_{\net}$ without measure theory, i.e.\ the Riemann integral is used.  
As it happens, $\MCT_{\net}$ is a special case of \cite{alivielvanzijn}*{19.36} where the limit function is not assumed to be continuous.   
Bourbaki proves a stronger version in \cite{boerbaki}*{IV.1, Theorem 1, p. 107}.
%Further note that the monotone convergence theorem \emph{for the Riemann integral} is proved in e.g. \cite{thomon3} using $\HBU$.
\begin{princ}[$\MCT_{\net}$]\label{crefi}
For continuous $f_{d}:(D\times I)\di \R$ forming an increasing net such that \(the limit\) $f=\lim_{d}f_{d}$ exists pointwise and is bounded and continuous, we have that $\lim_{d}\int_{0}^{1}|f(x)-f_{d}(x)|\,dx =0$.
\end{princ}
Note that we need $\WWKL$ to guarantee that the integral in $\MCT_{\net}$ exists, in light of \cite{sayo}*{Theorem 10}.  
Arzel\`a already studied the monotone convergence theorem (involving sequences) \emph{for the Riemann integral} in 1885, and this theorem is moreover proved in e.g. \cite{thomon3} using $\HBU$.
%The following theorem was first proved in \cite{samnets}*{\S3.3.1}.
\begin{thm}\label{merdes2}
The system $\RCAo+\WWKL+\MCT_{\net}$ proves $\WHBU$.
\end{thm}
\begin{proof}
In case $\neg(\exists^{2})$, all $F:\R\di \R$ are continuous by \cite{kohlenbach2}*{Prop.\ 3.12}.  
Hence, $\Q\cap [0,1]$ provides a \emph{countable} sub-cover for the canonical cover corresponding to $\Psi:E\di \R^{+}$.  By \cite{simpson2}*{X.1}, $\WWKL$ yields 
the sub-cover required for $\WHBU$.   

\smallskip

In case $(\exists^{2})$, suppose $\neg\WHBU$, i.e.\ there is some $ \Psi:[0,1]\di \R^{+}$, and $\eps_{0}>_{\R}0$, such that for all $y_{0}, \dots, y_{n}\in I$, the measure of $\cup_{i=0}^{n} I_{y_{i}}^{\Psi}$ is always below $ 1-\eps_{0}$. 
Now let $D$ be the set of finite sequences of reals in $I$ (without repetition) and define `$v \preceq w$' for $w, v\in D$ if $(\forall i<|v|)(v(i)\in w)$.
Clearly, $\preceq$ is transitive and reflexive, and also satisfies  
item~\eqref{bulk} in Definition \ref{nets}.   

\smallskip

Now define $f_{w}:I\di \R$ as follows:  if $w=\langle x\rangle$ for some $x\in I$, then $f_{w}$ is $0$ outside of $I_{x}^{\Psi}$, while inside the latter, $f_{w}(x)$ is the piecewise linear function that is $1$ at $x$, and $0$ in $x\pm\Psi(x)$.
If $w$ is not a singleton, then $f_{w}(x)=\max_{i<|w|}f_{\langle w(i)\rangle}(x)$.  

\smallskip

Then $f_{w}$ is increasing (in the sense of Definition \ref{inet}) and converges to the constant one function (in the sense of Definition \ref{convnet}), as for any $v\succeq \langle x \rangle$, we have $f_{v}(x)=1$.
Now, $\lim_{w}\int_{0}^{1}f_{w}(x)\,dx =1$ by $\MCT_{\net}$ and consider $1_{w}(x)\geq_{\R} f_{w}(x)$, where the erstwhile is the indicator function for $\cup_{i<|w|}I_{w(i)}^{\Psi}$ (and Riemann integrable on $I$).  Hence, there is $v_{0}$ such that $\int_{0}^{1}1_{v_{0}}(x) \, dx > 1-\eps_{0}$, and
as the left-hand side is the measure of $\cup_{i=0}^{|v_{0}|-1}  I_{v_{0}(i)}^{\Psi}$, we obtain a contradiction.  
Hence $\WHBU$ also follows in case $(\exists^{2})$, and we are done.  
\end{proof}
Since the $\ECF$-translation of $\MCT_{\net}$ readily follows from $\WWKL$, we cannot obtain $\HBU$ from this convergence theorem.
An equivalence $\MCT_{\net}\asa \WHBU$ seems desirable, but we do not know a proof at this point.  
The absence of almost any structure on the index sets in $\MCT_{\net}$ is perhaps the cause of all difficulties.  

\section{Computability theory and measure theory}\label{CTM}
In this  section, we study realisers for $\WHBU$ in computability theory.  
In particular, we construct such a realiser, denoted $\Lambda_{\bf S}$,  that does not add any extra power to the Suslin functional as in $(\SS^{2})$, in contrast\footnote{Realisers for the Heine-Borel theorem and Lindel\"of lemma are studied in \cites{dagsam, dagsamII,dagsamIII, dagsamV}.  It is shown that $\exists^{2}$ plus a realiser for Heine-Borel theorem computes a realiser for $\ATR_{0}$, while $\exists^{2}$ plus a realiser for the Lindel\"of lemma for Baire space computes the Suslin and Superjump functionals.} to the Heine-Borel theorem and the Lindel\"of lemma.  
We recall the definition of the Suslin functional:
\be\tag{$\SS^{2}$}
(\exists\SS^{2}\leq_{2}1)(\forall f^{1})\big[  (\exists g^{1})(\forall x^{0})(f(\overline{g}n)=0)\asa \SS(f)=0  \big].
\ee
We introduce realisers for $\WHBU$ and some definitions in Sections \ref{vintro} and \ref{vrelim}.
The construction of $\Lambda_{\bf S}$ may be found in Section \ref{constr}, as well as a proof that $\Lambda_{\bf S}+{\bf S}$ computes the same functions as the Suslin functional $\SS$. 
As an application, we show that $\FIVE^{\omega}+\WHBU $ does not  prove $\HBU$.
We also introduce a new hierarchy for second-order arithmetic involving $\Theta$ and $\HBU$ in Section \ref{kew}.
\subsection{Introduction: $\WHBU$ and its realisers}\label{vintro}
We discuss the brief history of realisers for $\WHBU$, list the associated definitions, and formulate the associated aim of this section in detail.

\smallskip

Now, the class of \emph{weak fan functionals}, or simply \emph{$\Lambda$-functionals}, was introduced in \cite{dagsam} and investigated further in \cite{dagsamII}. 
This class arose in the study of a version of \emph{weak weak K\"onig's lemma} from Nonstandard Analysis, 
but minor variations of $\Lambda$-functionals also provide us with realisers of some classical theorems (not involving Nonstandard analysis) such as Vitali's covering theorem for uncountable covers; see Section \ref{hintro}.
We shall make use of the following definition.  We recall Notation \ref{skim}, in particular that for a finite sequence $\sigma^{0^{0}}$ with length $k$, `$ f \in [\sigma]$' means  $\sigma = \overline{f}k$.
\begin{definition}[$\Lambda$-functional]\label{lambdagvd}
A functional $\Xi$ of type $2 \rightarrow (0 \rightarrow 1)$  is a $\Lambda$-\emph{functional} if whenever $F:C \rightarrow \N$ we have that $\Xi(F) = \{f_i\}_{i \in \N}$ is a sequence in $C$ such that $\bigcup_{i \in \N}[\bar f_iF(f_i)]$ has measure 1.
\end{definition}
Here $C$ is the Cantor space, identified with $\{0,1\}^\N \subseteq \N^\N$. If $s$ is a finite binary sequence, we let $[s]$ be the set of extensions of $s$ in $C$, as before.

\smallskip

In \cite{dagsam} we proved the existence of a $\Lambda$-functional $\Lambda_0$ without using the Axiom of Choice. In \cite{dagsamII} we showed that there is a $\Lambda$-functional $\Lambda_1$, called $\Lambda_{\exists^2}$ below, such that all elements in $C$ computable in $\exists^2+\Lambda_{\exists^2}$ are also computable in $\exists^2$. 

\smallskip

For $\Theta$ satisfying $\SFF(\Theta)$ from Section \ref{prelim2}, i.e.\ a realiser for the Heine-Borel theorem for uncountable covers, no such $\Theta$ is computable in $\Lambda_{\exists^2}$ and $\exists^2$ (\cite{dagsam, dagsamIII}). 
The aim of this section is to show that there is another $\Lambda$-functional, called $\Lambda_{\bf S}$ and defined in \eqref{lambda1337}, such that every function computable in $\Lambda_{\bf S}$ and $\bf S$ is computable from the Suslin functional $\bf S$. Since the Superjump is computable in $\bf S$ and any instance of $\Theta$ (\cite{dagsamII}*{\S4}),  it follows that no instance of $\Theta$ is computable in $\Lambda_{\bf S}$ and $\bf S$. 
\subsection{Background definitions and lemmas}\label{vrelim}
%We will assume full familiarity with  \cite{Suslin}. 
In this section, we will introduce lemmas and concepts, mainly from \cite{dagsamII}, that are needed in Section \ref{constr}. 
\begin{definition}
We let ${\bf m}$ be the standard  product measure on $C = \{0,1\}^\N$. 
\end{definition}
Since $C$ is trivially homeomorphic to any countable product of itself, we take the liberty to use ${\bf m}$ as the measure of any further product of $C$ as well.
We will use $A$, $B$ for such products and $X$, $Y$ and $Z$ for subsets of such products. All sets we (have to) deal with below are measurable, so we tacitly assume all sets are measurable.  
The following basic results of measure theory are used without reference.
\begin{proposition}[Basic measure theory] \label{measure}~
\begin{itemize}
\item[(a)] If $X_n \subseteq A$ and ${\bf m}(X_n) = 1$ for each $n \in \N$, then ${\bf m}(\bigcap_{n \in \N}X_n )= 1.$
\item[(b)] If $X \subset A \times B$ has measure 1, then 
\[\textstyle
{\bf m}(\{x \in A \mid {\bf m}(\{y \in B \mid (x,y) \in X\}) = 1\}) = 1.
\]
\end{itemize}
\end{proposition}
We shall make use of the general machinery on measure-theoretic uniformity for $\bf S$ from \cite{Suslin}, summarised in Definition \ref{def.unif} and Proposition \ref{prop.unif}. Our construction of $\Lambda_{\SS}$ in Section \ref{constr} will be an adjustment of the construction of  $\Lambda_{\exists^2}$ from \cite{dagsamII} to the computability theory of $\bf S$. The technical details of the constructions of $\Lambda_{\exists^{2}}$ and $\Lambda_{\SS}$ are quite similar, and we refer the interested reader to \cite{dagsamII}; we shall rather focus on the underlying intuition.  
\begin{definition}[$C$-sets]\label{def.unif} Let $\textup{\seq}$ be the set of finite sequences of integers.
\begin{itemize}
\item[(a)] A \emph{Suslin scheme} on a set $X$ is a map $s \mapsto P_s$ sending $s\in \textup{\seq}$ to $P_s\subseteq X$.
\item[(b)] If $ {\bf P} = \{P_s\}_{s \in {\rm \seq}}$ is a Suslin scheme, then define ${\bf A}({\bf P}) = \bigcup_{f \in \N^\N} \bigcap_{n \in \N} P_{\bar fn}.$
The functional {\bf A} is known as the \emph{Suslin operator}.
\item[(c)] The \emph{$C$-sets} in $\N^\N$ and related spaces are the elements of the least set algebra containing the open sets and being closed under the Suslin operator.
\end{itemize}
\end{definition}
The notion of $C$-set was first introduced in \cite{severeruss} and is also studied in e.g.\ \cites{hinman, seker, kleinburgerlijk}, a fact unknown to the authors before the beginning of 2020.
Indeed, in previous versions of this paper and \cite{Suslin}, $C$-sets were called `Suslin sets'.    
All $C$-sets have codes  in $\N^\N$ in analogy with the coding of Borel sets and the set of such codes has $\Pi^1_1$-complexity.
%We have followed the terminology from \cite{Suslin} here, but there are other classes named `Suslin sets' in the literature. 
The next proposition is proved in detail in \cite{Suslin}.
\begin{proposition}\label{prop.unif} \hspace*{10mm}
\begin{itemize}
\item[(a)] If $A \subseteq \N^\N$ is a $C$-set, then $A$ is computable in $\bf S$ uniformly in any code for $A$. There is a countable upper bound on the ordinal ranks of the computations needed to determine membership in $A$ from $\bf S$.
\item[(b)] If $A \subseteq \N^\N$ is computable in $\bf S$ and $f$ with computation-ranks bounded by the countable ordinal $\alpha$, then $A$ is a $C$-set and there is a code for $A$ computable in $\bf S$, $f$ and any 
$\N^\N$-code for $\alpha$.
\item[(c)] If $A \subset C$ is a $C$-set, then ${\bf m}(A)$ is computable in $\bf S$ and a code for $A$.
\item[(d)] If $A \subseteq C$ is computable in $\bf S$ and ${\bf m}(A) > 0$, then $A$ contains an element computable in $\bf S$. This basis theorem can be relativised to any $f \in \N^\N$.
\item[(e)] The algebra of $C$-sets is a $\sigma$-algebra, i.e. closed under countable unions and complements, and thus contains the Borel sets.
\end{itemize}
\end{proposition}
We let $\omega_1^{{\bf S},g}$ be the first ordinal that is not computable in $\bf S$ and $g$, while $\omega_{1}^{\bf S}$ is $\omega_{1}^{\bf S, \emptyset}$. We also need the following result from \cite{Suslin}.
\begin{proposition}\label{propmeas} The set $\{g \in C:\omega_1^{\bf S} = \omega_1^{{\bf S},g}\}$ has measure 1.
\end{proposition}
Unless specified otherwise, the sets $X,Y, Z$ considered below are computable in ${\bf S}$, possibly from parameters and at a countable level, i.e.\ they are $C$-sets. Without pointing this out every time, we make use of the following result from \cite{Suslin}.
\begin{proposition} 
If $A$ and $B$ are measure-spaces as above and $X \subseteq A \times B$ is a $C$-set, then $\{x \in A \mid {\bf m}(\{y \in B \mid (x,y) \in X\}) = 1\}$
is a $C$-set with a code computable from ${\bf S}$ and any code for $X$. 
\end{proposition}
%Recall that we in \cite{Suslin} identify sets computable in ${\bf S}$ and $f$ at a level computable in ${\bf S},f$ with the class of Suslin sets with codes computable in ${\bf S},f$.
%

\begin{nota}The following notational conventions are used below: 
\begin{enumerate}
\renewcommand{\theenumi}{\alph{enumi}}
\item $a$, $b$, $c$ are numerical arguments or values in computations, while  $\vec{a}, \vec{b}, \vec{c}$ are finite sequences of such.
\item $i$, $j$, $n$, $m$ are integers for other purposes, such as indexing.
\item $f$, $g$, $h$ are elements of $C$, with finite sequences denoted $\vec{f}, \vec{g}, \vec{h}$.
\item $(f)$ denotes an infinite sequence $(f) = \{f_i\}_{i \in \N}$ from $C$. %We let $(f)^{\leq n}$ and $(f)^{>n}$ be the restrictions of $(f)$ to the indicated domains.
\item Given $(f)$ and $\vec c = (c_0 , \ldots ,c_{n-1})$, we define $(f)_{\vec c} := (f_{c_0}, \ldots , f_{c_{n-1}})$. %and we let $u(\vec c)$ be the maximal element in $\vec c$.
\end{enumerate}
\end{nota}
\begin{definition} Let $F:Y \rightarrow \N$ where $Y \subset C$, and let $(f)$ be as above.
\begin{itemize}
\item[(a)] We say that $(f)$ is \textbf{\emph{sufficient}} for $F$ if $F(f_i)$ is defined for all $i$ and the set $\bigcup_{i \in \N}[\bar f_iF(f_i)]$ has measure 1.
\item[(b)] We say that $(f)$ \textbf{\emph{fails}} $F$ if $F(f_i)$ is undefined for some $i $.
\end{itemize}
\end{definition}
The intuition behind this definition is that $(f)$ is sufficient for $F$ if $(f)$ is an acceptable value for $\Lambda(F)$. 
The key to the construction of $\Lambda_{\bf S}$ is that if we have a parametrised family of total functionals $F_x$, then we can use the same value for $\Lambda(F_x)$ for almost all $x$; moreover, if this family is computable in $\bf S$, we may choose this common value to be computable in $\bf S$. We shall construct approximations to $\Lambda_{\bf S}$ by recursion on the ordinals below $\omega_1^{\bf S}$, and the property of \emph{failing} will be used as a technical means to avoid `wild cases' of termination relative to our end product.

\smallskip
The following lemma, that is trivial from the point of view of measure theory, makes our intuition precise and plays an important part in the construction of $\Lambda_{\bf S}$.  
Note that we use commas to denote concatenations of finite sequences from $C$, while we use $\langle \cdot \rangle$ when we consider sequences involving other kinds of objects. A detailed proof can be found in \cite{dagsamII}.
\begin{lemma}\label{hovedlemma}
Let $A$ be a finite product of $C$, and let $\vec g$ range over the elements of $A$. Let $\vec c$ be a non-repeating sequence of integers of length $k$, and let $F:Z \rightarrow \N$ where $Z \subseteq C \times C^{k}\times A$ is a measurable set.
If $Y \subseteq C^{k} \times A$ has measure 1, then
\begin{itemize}
\item[(a)] the set of $\langle(f),\vec g\rangle$ such that $(f)_{\vec c},\vec g \in Y$ has measure 1,
\item[(b)] the measure of the following set is $1$\textup{:} the set of $\langle(f),\vec g\rangle$ such that either
\[
\text{${\bf m}(\{f \mid f, (f)_{\vec c},\vec g \in Z\}) = 1$  and $(f)$ is sufficient for $\lambda f.F(f,(f)_{\vec c},\vec g)$},
\]
\begin{center}
or
\end{center}
\[
\text{${\bf m}(\{f \mid f, (f)_{\vec c},\vec g \in Z\}) < 1$ and $(f)$ fails $\lambda f.F(f,(f)_{\vec c},\vec g)$}.
\]
\end{itemize}
\end{lemma}
The conclusions of the lemma do not change if we restrict $(f)$ to sequences from a subset $X$ of $C$ of measure 1. The requirement that $\vec c$ is non-repeating is essential, since otherwise the set of $(f)_{\vec c}$ will have measure 0.

\subsection{The construction of a weak $\Lambda$-functional} \label{constr}
We construct the $\Lambda$-functional $\Lambda_{\bf S}$ and show that $\Lambda_{\bf S}+{\bf S}$ computes the same functions as $\bf S$. 
\smallskip

The following partial ordering is crucial to our construction of $\Lambda_{\bf S}$.
\begin{lemma}\label{lemma.useful}
There is a well-ordering $(A,\prec)$ of a subset of $\N$ of order type $\omega_1^{\bf S}$, semi-computable in $\bf S$, such that for each $a \in A$, $\{\langle b,c\rangle \mid b \prec c \prec  a\}$ is computable in $\bf S$, uniformly in $a$.
\end{lemma}
\begin{proof}
We let $A$ be the set of computation tuples $a = \langle e,\vec a , b\rangle$ such that $\{e\}({\bf S},\vec a) = b$ with the norm $|| \cdot ||$, and we let $a \prec a'$ if $||a|| < ||a'||$, or if $||a|| = ||a'||$ and $a < a'$.
This ordering has the desired properties. 
\end{proof}
We introduce some more notation.
\begin{nota}{\em
We  let $[f] $ denote a family $[f] = \{(f_a)\}_{a \in A} = \{\{f_{a,i}\}_{i \in \N}\}_{a \in A}.$
 When we have fixed $[f]$, and $a \in A$, we let $[f]_a$ be $[f]$ restricted to $\{b \in A \mid b \preceq a\}$ and we let $[f]_{\prec a}$ be $[f]$ restricted to $\{b \in A \mid b \prec a\}$.
}\end{nota}
\begin{definition}{\em  Let $[f]$ be as above. We define $\Lambda_{[f]}$ as the partial functional, accepting partial functionals of type 2 as inputs, as follows: $\Lambda_{[f]}(F) := \{f_{a,i}\}_{i \in \N}$ if
\begin{itemize}
\item[(a)] $(f_a)$ is sufficient for $F$,
\item[(b)] $F(f_{b,j})$ is defined for all $b \preceq a$ and all $j \in \N$,
\item[(c)] $(f_b)$ is not sufficient for $F$ for any $b \prec a$.
\end{itemize}
Similarly, $\Lambda_{[f]_a}$ and $\Lambda_{[f]_{\prec a}}$ are defined by replacing $[f]$ with $[f]_a$ or $[f]_{\prec a}$.
}\end{definition}
Since the specification for a $\Lambda$-functional only specifies the connection between $F$ and $\Lambda(F)$, and does not relate $\Lambda(F)$ and $\Lambda(G)$ for different $F$ and $G$, and since there is at least one $\Lambda$-functional $\Lambda_0$ (see \cite{dagsam}), functionals of the form $\Lambda_{[f]}$ can be extended to total $\Lambda$-functionals. We trivially have the following. 
\begin{lemma} 
Let $[f]$ be as above and also partially computable in ${\bf S}$ . Then every function $g$ computable in $\Lambda_{[f]}$ is also computable in ${\bf S}$.
\end{lemma}
To obtain our main results, we must construct $[f]$ such that any function $g$ computable in any total extension of $\Lambda_{[f]}$ is still computable in ${\bf S}$. 
By convention, we let $C^0$ consist of a singleton, the empty sequence,  with measure 1.

\medskip

The following two lemmas are closely related to resp.\ \cite{dagsamII}*{Lemmas~3.29 and~3.30}.  Indeed, the proof of Lemma \ref{lemma.ind} proceeds via a line-to-line translation from the $\omega_1^{\rm CK}$-recursion in the proof of \cite{dagsamII}*{Lemma 3.29} to an $\omega_1^{\bf S}$-recursion. 
Since the proof is long and technical with no new additions, we restrict ourselves to an outline of the proof here. 
Whenever the proof makes use of \cite{dagsamII}*{Proposition 3.23}, the proof of Lemma \ref{lemma.ind} makes use of Proposition \ref{prop.unif} instead. 
Since the proof of Lemma~\ref{lemma.sup} is based on details in the proof of Lemma \ref{lemma.ind}, we refer to \cite{dagsamII}*{Lemma 3.30} for the exact argument.

\smallskip

In the formulation of the latter lemma, we have taken the (notational) liberty to ignore other ways of listing the inputs. There is no harm in this since we may always use Kleene's S6 to permute inputs. 
Our motivation is that stating and proving the general result will be much more cumbersome, but all genuine mathematical obstacles are however gone. 
\begin{lemma}\label{lemma.ind}
By ${\bf S}$-recursion on $a \in A$, we can construct $[f] = \{(f_a)\}_{a \in A}$ and sets $X_{a,k} \subseteq C^k$ of measure 1 \(for each $k \in \N$ and $a \in A$\) such that an alleged computation $\{e\}\big(\Lambda_{[f]_a},{\bf S},\vec a  ,\vec h, \vec g\big)$ will terminate whenever the parameters satisfy:
\begin{enumerate}
\item[(a)] $a \in A$ has norm $\alpha = ||a||$, $e$ is a Kleene-index, $\vec a\in \textup{\seq}$, and $\vec g \in X_{a,k}$, 
\item[(b)] $\vec h$ is a sequence from $\{f_{b,i} \mid i \in \N \wedge b \preceq a\}$,  
\item[(c)] there is some extension $[f']$ of ${[f]_a}$ such that $\{e\}(\Lambda_{[f']}, {\bf S},\vec a , \vec h, \vec g)\!\!\downarrow$ with a computation of ordinal rank at most $\alpha$. 
\end{enumerate} 
\end{lemma}
\begin{proof}(Outline)
\begin{itemize}
\item Given $a \in A$, we assume that $[f]_{\prec a}$ is constructed, that all sets $X_{b,k}$ are constructed for $b \prec a$, that the induction hypothesis holds and that what is constructed so far, is computable in ${\bf S}$.
\item The main step is, for each $k$, to construct  a $C$-set $Z_{a,k} \subseteq C^{\omega} \times C^k$ of measure 1 such that for all $\langle(f),\vec g \rangle \in Z_k$, the property stated in the Lemma will hold for $\vec g$ and $a$, except the requirement that $(f)$ is computable in ${\bf S}$, if we extend $[f]_{\prec a}$ with $(f)$.
\item We can then use Proposition \ref{measure}, Proposition \ref{prop.unif}, and Gandy selection to find an $(f_a)$ computable in ${\bf S}$ such that the following set has measure one:
\[
X_{a,k} = \{\vec g \in C^k \mid \langle(f_a),\vec g\rangle \in Z_{a,k}\}.
\]
\item The set $Z_{a,k}$ is the intersection of countably many sets needed to handle each case given by the index $e$ and by how the sequence $\vec h$ is selected from $[f]_{\prec a}$ and $(f_a)$. We use Lemma \ref{hovedlemma} to handle the cases corresponding to Kleene's S8, ensuring that we can use $(f_a)$ as the value of $\Lambda_{\bf S}(F)$ for all $F$ semi-computable in ${\bf S}$ and $\Lambda_{\bf S}$ and total on a set of measure 1 exactly at stage $||a||$.
\item The whole construction is tied together using the recursion theorem for computing relative to ${\bf S}$
\end{itemize}
This finishes the proof of the lemma.
\end{proof}
 For each $k\in \N$, define $X_k = \bigcap_{a \in A} X_{a,k}$. These sets are  $C$-sets, but not with a code computable in ${\bf S}$. They are complements of sets semi-computable in ${\bf S}$.
 For a proof of the following lemma, see (the proof of) \cite{Suslin}*{Lemma 3.27}.
 \begin{lemma}\label{lemma.sup}
 For each $k$ and $\vec g \in X_k$ we have ${\bf m}(\{g \mid g,\vec g \in X_{k+1}\}) = 1$.
 \end{lemma}
 Let $\Lambda_0$ be the $\Lambda$-functional constructed in \cite{dagsam}, and let $[f]$ be as constructed in the proof of Lemma \ref{lemma.ind}. We define $\Lambda_{\bf S}$ as follows:
 \be\label{lambda1337}
 \Lambda_{\bf S}(F) = \left \{ \begin{array}{llc} \Lambda_{[f]}(F)& {\rm if\;defined}\\\Lambda_0(F)& {\rm otherwise}\end{array}\right.
 ,
 \ee
 and prove our main theorem as follows. 
 \begin{theorem}\label{theorem.lambda}
 If $f:\N\di \N$ is computable in $\Lambda_{\bf S}+\bf S$, then it is computable in ${\bf S}$.
 \end{theorem}
 \begin{proof}
 We shall prove the stronger claim \eqref{calem} below by induction on the length of the computation.  
 We need some notation as follows.
 Let $e$ be a Kleene index, let $\vec a$ be a sequence from $\N$, let $\vec h$ be a sequence from $[f]$, and let $\vec g$ of length $k$ be a sequence from $\bigcap_{a \in A}X_{a,k}$ such that $\omega_1^{\bf S} = \omega_1^{{\bf S} ,\vec g}$. 
 By Proposition \ref{propmeas}, the final restriction does not alter the measure of the set.
Now consider the following claim:
\be\label{calem}
\{e\}(\Lambda_{\bf S} , {\bf S} , \vec a , \vec h , \vec g) = b\di (\exists a\in A) (\{e\}(\Lambda_{[f]_a},{\bf S},\vec a , \vec h , \vec g) = b)
\ee 
The theorem follows from the claim \eqref{calem} and the total instances $\lambda c. \{e\}(\Lambda_{\bf S},c)$. %Note that the requirement that $\vec h$ and $\vec c$ are repetition-free used in the proof are vacuous for these cases, i.e.\ there is no additional restriction of the main result.
 
\smallskip 
 
 We now prove the claim \eqref{calem} by induction on the ordinal rank of the computation $\{e\}(\Lambda_{\bf S} , {\bf S} , \vec a , \vec h , \vec g) = b.$ 
 The proof is split into cases according to which Kleene-scheme $e$ represents, and all cases except those for application of ${\bf S}$ or $\Lambda_{\bf S}$ are trivial. 
 Thus, we (only) consider the two cases \eqref{cake1} and \eqref{cake2}.  First, we consider
 \be\label{cake1}
 \{e\}(\Lambda_{\bf S} , {\bf S}, \vec a , \vec h , \vec g) = {\bf S}(\lambda s.\{e_1\}(\Lambda_{\bf S} , {\bf S} , s,\vec a,\vec h, \vec g)),
\ee 
which yields the following by the induction hypothesis:
 \[
( \forall s \in {\rm \seq}) (\exists a \in A) [\{e_1\}(\Lambda_{[f]_a},s,\vec a , \vec h , \vec g)\!\!\downarrow].
 \]
 Since $\omega_1^{\bf S}$ is $\vec g$-admissible, there is a bound on how far out in $A$ we need to go, i.e.\ $(\exists a \in A) (\forall s \in {\rm \seq} )[\{e_1\}(\Lambda_{[f]_a},s,\vec a , \vec h , \vec g)\!\!\downarrow]$, 
 and $\{e\}\big(\Lambda_{[f_a]} , {\bf S}, \vec a , \vec h , \vec g\big)\!\!\downarrow$ follows.
 
\smallskip
\noindent
For the second case, consider 
\be\label{cake2}
\{e\}(\Lambda_{\bf S},{\bf S} , \vec a,\vec h , \vec g) = \Lambda_{\bf S}(\lambda g. \{e_1\}(\Lambda_{\bf S} , {\bf S},\vec a , \vec h , g , \vec g)).
\ee
By Lemma \ref{lemma.sup} and the induction hypothesis, for almost all $g$ there is an $a_g \in A$ such that $\{e_1\}\big(\Lambda_{[f]_{a_g}},{\bf S},\vec a , \vec h , g , \vec g\big)\!\! \downarrow$.
Now consider the sequence 
\[
a \mapsto {\bf m}(\{g \mid \{e_1\}_{||a||}(\Lambda_{[f]_a} , {\bf S} , \vec a , \vec h ,g, \vec g)\!\!\downarrow\}),
\]
where $\{e_1\}_{||a||}$ means that we only consider computations with ordinal rank below $||a||$.  This sequence is increasing, computable in $\vec g$ and ${\bf S}$, and has limit 1. 
By the assumption on $\vec g$, we have $\omega_1^{\bf S} = \omega_1^{{\bf S},\vec g}$ and this ordinal is $\vec{g}$-admissible.  As a result, there is $a \in A$ such that 
\[
{\bf m}(\{g \mid \{e_1\}_{||a||}(\Lambda_{[f]_a }, {\bf S} , \vec a , \vec h ,g, \vec g)\!\!\downarrow\}) = 1. 
\] 
Moreover, we may assume that $\vec h$ is in $[f]_a$.  Hence, our construction guarantees that $(f_a)$ is sufficient for $\lambda g. \{e_1\}(\Lambda_{\bf S} , {\bf S},\vec a , \vec h , g , \vec g)$, unless some $(f_b)$ already does the job for $b \prec a$.
We may conclude that 
\[ 
\{e\}(\Lambda_{\bf S},{\bf S} , \vec a , \vec h , \vec g) = (f_b) = \{e\}( \Lambda_{[f]_a},{\bf S} , \vec a , \vec h , \vec g)
\] 
for some $b \preceq a$. This ends the induction step, and we are done.  
\end{proof}
As an application, we now use the functional $\Lambda_{\bf S}$ to construct a model for $\FIVE^{\omega}+\WHBU$ in which $\neg\HBU$ holds. We make use of the well-known fact that all terms in G\"odel's $T$ have set-theoretical interpretations as elements of the maximal type-structure of total functions, and that this is also the case when we use extra constants for total functionals. Recall that $\mu^{2}$ is Feferman's search operator from Section \ref{prelim2}.  We need the following definition. 
\begin{definition}~
\begin{enumerate}
\renewcommand{\theenumi}{\alph{enumi}}
\item Let $\SUS$ be the set of functions computable in $\bf S$. 
\item For any type $\sigma$, let ${\rm \SUS}_\sigma$ be the set of functions of type $\sigma$ in the maximal type-structure definable via a term in G\"odel's T using constants for $\mu^{2}$, $\SS$, $\Lambda_{\SS}$ and elements in $\SUS$.
\item We define a \emph{partial equivalence relation} `$\sim_\sigma$' on ${\rm \SUS}_\sigma$ by recursion on $\sigma$:
\begin{itemize}
\item[(i)] The relation $\sim_{\N}$ is the identity relation on $\N$
\item[(ii)] If $\sigma = \tau\rightarrow \delta$ and $F_1,F_2\in {\rm \SUS}_\sigma$, we let $F_1 \sim_\sigma F_2$ if $\phi_1 \sim_{\tau}\phi_2\di F_1\phi_1 \sim_{\delta} F_2\phi_2$ for all $\phi_{1}, \phi_{2}\in{\rm \SUS}_{\tau}$.
\end{itemize}
\end{enumerate}
\end{definition}
By the properties of $\Lambda_{\bf S}$, we observe that ${\rm \SUS}_{0 \rightarrow 0} = {\rm \SUS}$. We also observe that $\sim_{0 \rightarrow 0}$ is the identity relation on $\SUS$. It is then easy to see that all $\sim_\sigma$ are partial equivalence relations, i.e.\ transitive and symmetric, and thus equivalence relations on the set of self-equivalent functionals $F$. We may then form the \emph{Hereditarily Extentional Collapse} (\HEC), intimately related to the Mostovski collapse from set theory. 
For numerous applications of the $\HEC$, we refer to \cite{longmann}.

\smallskip
We state the following well-known fact without proof.
\begin{lemma}
Seen as a model for $\Z_{2}$, $\SUS$ satisfies $\Pi^1_1$-comprehension, and all $\Pi^1_1$-formulas with parameters from $\SUS$ are absolute for $({\rm \SUS}, \N^\N)$.
\end{lemma}
Our next lemma will be important for showing that applying the $\HEC$ indeed produces a model of $\RCAo$.  
\begin{lemma} 
If $F \in {\rm \SUS}_\sigma$, then $F \sim_\sigma F$.
\end{lemma}
\begin{proof} The proof is by induction on the term defining $F$. More precisely, we have to prove that if $t$ is a term of type $\sigma$, with free variables among $x_1^{\tau_1}, \ldots , \ldots , x_n^{\tau_n}$, then the interpretation $[[t]]$, seen as an element of $\tau_1 \rightarrow \cdots \rightarrow \tau_n \rightarrow \sigma$,  will be self equivalent. Due to the possible use of $\lambda$-closure, $[[t]]$ is in our $\SUS$-hierarchy.

\smallskip

The induction base deals with the variables and constants, including all the recursor constants ${\rm \mathbf{R}}_\sigma$. Of these, it is only the constant for $\Lambda_{\bf S}$ that is nontrivial (and new in this case). The induction steps consist of terms formed by $\lambda$-abstraction and by application,  and are trivial (and well-known).
So, let $F$ and $G$ be in ${\rm \SUS}_{(0 \rightarrow 0) \rightarrow 0}$ be $\sim$-equivalent, meaning, in this case, that they are equal on ${\rm \SUS}$. Since they are $T$-definable from $\Lambda _{\bf S}$, ${\bf S}$, $\mu^{2}$ and elements from ${\rm \SUS}$, they are in particular Kleene computable from $\Lambda_{\bf S}$ and ${\bf S}$. Hence, we have 
\be\label{Form}
\Lambda_{\bf S}(F) = \Lambda_{[f]}(F) = \Lambda_{[f]}(G) = \Lambda_{\bf S}(G), 
\ee
where $[f]$ and $\Lambda_{[f]}$ are as constructed in the proof of Lemma \ref{lemma.ind}. 
In light of \eqref{Form}, we may conclude $F \sim_{( 0\rightarrow 0) \rightarrow 0} G$ and the proof is finished.
\end{proof}
Finally, the following theorem readily follows.
\begin{theorem}\label{Frink}
The system $\FIVE^{\omega}+\WHBU$ cannot prove $\HBU$.
% ${\rm RCA}_0^\omega + \Pi^1_1-{\rm CA}_0 + {\rm WHBU}$ does not prove ${\rm HBU}$ 
\end{theorem}
\begin{proof}
We construct a model $\M$ of $\FIVE^{\omega}+\WHBU+\neg\HBU$.
This model $\M$ is the aforementioned \emph{hereditarily extensional collapse}  $\{{\rm \SUS}_{\sigma}/\sim_{\sigma}\}_{\sigma\;\;{\rm \type}}$, seen as a type structure. Clearly, $\M$ satisfies $(\SS^{2})$ as $\SUS$ is unchanged under the collapse; $\M$ satisfies ${\RCA}_0^\omega$ since $\M$ is a model of G\"odel's $T$ and we use $\mu^{2}$ to ensure that $\QFAC^{1,0}$ is satisfied.  The model $\M$ satisfies ${\rm \WHBU}$ since the collapse of $\Lambda_{\bf S}$ is a realiser for ${\rm \WHBU}$ within $\M$. 

\smallskip

Next, we show that $\HBU$ fails in $\M$.  To this end, we consider \cite{dagcie18}*{Theorem~1.c)} which establishes that the functional $\Gamma$, the realiser for non-monotone inductive definitions, is computable in any $\Theta$-functional and $\SS^{2}$. 
Now, Richter proves in \cite{vijfopdeschaal} that even inductive definitions given by arithmetical functionals $F:C\di C$
 (actually $\Pi^0_2$ suffices) can have closure ordinals beyond the first recursively Mahlo ordinal; these can therefore be used to construct $f\in C$ not computable in ${\bf S}^{2}$. 

\smallskip

Let $F:C \rightarrow C$ be one such arithmetical functional.  
In the proof of \cite{dagcie18}*{Theorem~1.c)}, a functional $G:C \rightarrow C$ is defined in terms of $F$, $\mu$, and $\SS$ via a term of G\"odel's $T$.  This functional defines an open covering of $C$ and its key property is that whenever $f_1 , \ldots , f_n$ defines a finite sub-covering, then the set inductively defined from $F$ is computable in $f_1 , \ldots , f_n,{\bf S}$, and in fact definable from $f_1 , \ldots , f_n , {\bf S}, \mu^{2}$ by a term of G\"odel's $T$. 
Since $F$ is chosen as an arithmetical functional such that the fixed point of the associated inductive definition is not in $\M$, while $G$ is in $\M$, we must have that one of $f_1 , \ldots , f_n$ is outside $\M$, and $\HBU$ fails in the latter.
\end{proof}
We conjecture that $\Z_2^\omega+\WHBU$ cannot prove $\HBU$, but have no (idea of a) proof.
\subsection{Two new hierarchies relating to second-order arithmetic}\label{kew}
We have previously shown that the combination $\mu^{2}+\Theta^{3}$ computes a realiser for $\ATR_{0}$, while the latter schema (not involving a realiser) is provable in $\ACAo+\HBU$ (\cite{dagsam, dagsamII, dagsamIII, dagsamV}).  We refer to this phenomenon as an `explosion' as both components are weak (in isolation) compared to the combination.   
The aim of this section is to exhibit a number of similar explosions that provide a sweeping generalisation of the aforementioned results, yielding new hierarchies parallel to the usual hierarchy of second-order arithmetic based on comprehension. 
A similar parallel hierarchy is described in \cite{shoma}, but based on the \emph{axiom of determinacy} from set theory, while we work with the -more natural in our opinion- theorem $\HBU$ and its realiser $\Theta$. 

\smallskip
%\subsubsection{A hierarchy based on Heine-Borel compactness}

First of all, as expected, a central role is played by transfinite recursion, which we now define.  Let $\WO(X)$ express that $X$ is a countable well-ordering as in \cite{simpson2}*{V.1.1}.  The following definition may be found in \cite{simpson2}*{VI.7}.  
To be absolutely clear, $\theta$ below is part of $\L_{2}$: no type two parameters are allowed.  
\bdefi[$\PITK$]
For any $\theta\in \Pi_{k}^{1}$ and $X\subseteq \N$, we have
\[
\WO(X)\di (\exists Y\subseteq \N)H_{\theta}(X, Y), 
\]
where $H_{\theta}(X, Y)$ states that $X$ is a linear ordering and that $Y=\{(m,j): j\in \textsf{\textup{field}}(X)\wedge \theta(n, Y^{j})  \}$ for $Y^{j}:=\{ (m,i): i<_{X}j \wedge (m, i)\in Y  \}$. 
\edefi
We note the unfortunate use of `$\theta$' for an $\L_{2}$-formula, and `$\Theta$' for special fan functionals.  
By \cite{simpson2}*{Table 4}, $\PITK$ is strictly between $\SIXk$ and $\Pi_{k+1}^{1}\textup{\textsf{-CA}}_{0}$.
%NEW: changed SIXK to SIXk in the previous line.

\smallskip

Secondly, we prove the following two theorems. 
\begin{thm}
Uniformly for each instance of $\Theta^{3}$, there is a type three functional $\TR(F,A,< )$, where $F:2^\N \di 2^\N$  and $<_{A}$ is a binary relation on $A  \subseteq \N$, such that if $(A,<_{A})$ is a well-ordering and $a \in A$ then 
\be\label{KIT}
\TR(F,A,<_{A})(a) = F\big(\{\langle b,c\rangle \in A\times A : b <_{A} a \wedge c \in \TR(F,A,<_{A})(b)\}\big), 
\ee
i.e.\ $ \TR(F,A,<_{A})(a) $ is the result of iterating $F$ along $(A, <_{A})$ up to $a$.
\end{thm}
\begin{proof}
Immediate by \cite{dagsamII}*{Cor.\ 3.16}.
\end{proof}
A realiser for $\PITK$ is a functional that takes as input $X\subseteq \N$, a binary relation $<_{X}$, and $f:\N\di \N$, and outputs $Y\subseteq \N$ such that $H_{\theta}(X, Y)$ if $\WO(X)$ and $\theta$ is the $\Pi_{k}^{1}$-formula in Kleene normal form with  $f(x_{1},\dots,  x_{k}, n)=0$ providing the innermost quantifier-free part. 
\begin{cor}
The combination $\SS_{k}^{2}+\Theta$ computes a realiser for $\PITK$. 
\end{cor}
\begin{proof}
This follows from the theorem by recalling that $\SS_{k}^{2}$ decides $\Pi_{k}^{1}$-formulas.  
\end{proof}
Finally, we have the following theorem. 
 %Question: the previous corollary should also say: `via a term of G\"odel's $T$'?
\begin{thm}
The system $\SIXK+\HBU$ proves $\PITK$. 
\end{thm}
\begin{proof}
Our proof proceeds via contradiction: fix $X\subseteq \N$ such that $\WO(X)$ and suppose we have $(\forall Y\subset\N)(\exists k^{0})\neg H_{\theta}(k, X, Y)$ for some $\L_{2}$-formula $\theta \in \Pi_{k}^{1}$.  
Here, $H_{\theta}(i, X, Y)$ is just $H_{\theta}(X, Y)$ with the additional restriction $j<_{X}k$, as can be found in \cite{simpson2}*{V.2.2}.  Clearly, $H_{\theta}(X, Y)$ is decidable given $\SS_{k}^{2}$ and $\QFAC^{1,0}$ applied to
$(\forall Y\subset\N)(\exists k^{0})\neg H_{\theta}(k, X, Y)$ yields some $G:C\di \N$.  Use $\HBU_{\c}$ to obtain $f_{1}, \dots, f_{k_{0}}\in C$ such that $\cup_{i\leq k_{0}}[\overline{f_{i}}G(f_{i})]$ still covers $C$. 
We now note that given $\SS_{k}^{2}$, we can always apply transfinite recursion `once more', i.e.\ given $k$ and $Y$ such that $H(k, X, Y)$, we can define $Z\subset \N$ such that $H(k', X, Z)$, where $k'$ is the least number above $k$ according to $<_{X}$.
The same holds for a finite number of iterations via $\IND$.  
Now consider $k_{1}:=\max_{i\leq k_{0}}G(f_{i})$ and note that we have obtained a contradiction; $\IND$ can be avoided by letting $G^{2}$ be the least number as above.  
 \end{proof}
 \section{Discussion and conclusion}\label{konkelfoes}
 We discuss two observations (Section \ref{per}) and some foundational musings (Section \ref{muse}) pertaining to our results.  
 \subsection{Two observations}\label{per}
 We dicuss the possibility for a template based on our results, and an interesting observation dubbed \emph{dichotomy phenomenon}.
 \subsubsection{Towards a template}\label{tempy}
 The proof of Theorem \ref{theorem.lambda} is similar to a proof of the existence of $\Lambda_{\exists^2}$, to such an extent that we just gave reference to that proof for many of the technicalities. Based on the theories on measure-theoretic uniformity for recursion in $\exists^2$ and in ${\bf S}$, the main constructions follow the same pattern.
% Even though the exposition is different (and hopefully better) from that in \cite{dagsamII}, it is basically the same construction.
  The question is how much further this kind of construction could lead.
 
 \smallskip
 
The results on measure-theoretic uniformity turned out to be quite similar for computability in $\exists^2$ and in ${\bf S}$.  In each case, the measure theory of subsets of the continuum computable in $\exists^2$ or ${\bf S}$ can be handled \emph{within} the class of functions computable in $\exists^2$ and ${\bf S}$ via suitable coding. It seems unlikely that something like this can be relativised to \emph{all} functionals of type 2. In this light, we offer the following open problem.
 \begin{problem} 
 Is there a functional $F$ of type 2 such that for all $\Lambda$-functionals $\Xi$ there is a $\Theta$-functional computable in $F$ and $\Xi$?
 \end{problem}
 We conjecture the answer to be negative, but see no way to establish this.
 
 \subsubsection{A dichotomy phenomenon}\label{dich}
The main result of this section is another example of a `dichotomy' phenomenon that we have observed during the study of functionals arising from classical theorems, namely as follows.

\smallskip

On one hand, \emph{positive results} about relative dependence are of the form that elements in one class of functionals can uniformly be defined from elements in another class of interest \emph{via a term in a small fragment} of G\"odel's $T$.
On the other hand, \emph{negative results} are of the form that there is one element $\Phi$ in one class such that no element $\Psi$ in the other class is \emph{computable} in $\Phi$ in the sense of Kleene, often even not relative to any object of lower type. 

\smallskip

We find this to be an interesting observation, and a source for classification of the (computational) strength of theorems.

\subsection{Foundational musings}\label{muse}
We discuss the foundational implications of our results, which we believe to be rather significant \emph{and} different in nature from \cite{dagsamIII, dagsamV}.  

\smallskip

As noted above, the development of measure theory in `computational' frameworks like e.g.\ Reverse Mathematics, constructive mathematics, and computable analysis, 
proceeds by studying the computational properties of \emph{countable approximations} of measurable objects.  To be absolutely clear, theorems in these fields are generally \emph{not} 
about objects \emph{themselves}, but about \emph{representations} of  objects. Of course, this observation is of little concern \emph{in general} as there are `representation theorems' that express
that `nice' representations always exist.  Nonetheless, there are two conceptual problems that arise from our results, as follows. 

\smallskip

First of all, \emph{in the particular case} of RM, there is a potential problem with using representations: the aim of RM is to find the minimal axioms required to prove theorems of ordinary mathematics `as they stand' (see \cite{simpson2}*{I.8.9.5} for this exact wording).  Thus, the logical strength/hardness 
of a theorem should not change upon the introduction of representations, lest this distort the RM-picture!
However, we have identified interesting theorems, i.e.\ the Vitali covering theorem and $\WHBU$, for which the hardness changes quite dramatically upon introducing codes.  
%in Section \ref{conv} provide two examples of theorems, namely the \emph{dominated and monotone convergence theorems}, for which the hardness changes quite dramatically upon introducing codes.  
Indeed, in terms of (conventional) comprehension, $\WHBU$ is not provable in $\Z_{2}^{\omega}$ but provable $\Z_{2}^{\Omega}$; the same holds if we restrict to Baire 2 or semi-continuous functions (see Footnote \ref{floeker}).
However, $\ATR_{0}+\Delta_{2}^{1}$-induction proves $\HBU$ (and hence $\WHBU$) formulated via codes for Borel functions by \cite{basket2}*{Prop.\ 3.2}.  

\smallskip

Secondly, there is another, more subtle, aspect to our results, namely pertaining to the formalisation of mathematics in second-order arithmetic.  % is capable of formalising mathematics.  
Simpson (and many others) claims that the latter can accommodate large parts of mathematics:
\begin{quote}
[\dots] focusing on the language of
second order arithmetic, the weakest language rich enough to express and develop
the bulk of mathematics. (\cite{simpson2}*{Preface})
\end{quote}
Let us first discuss a concept for which the previous quote is undeniably correct: continuous functions, which are represented by codes in RM (see \cite{simpson2}*{II.6.1}).  
%It is then a natural question whether codes actually capture \emph{all} continuous functions (say in a weak system).  Indeed, if codes only captured a special sub-class, then 
%a theorem of RM would be about that sub-class, and not about all continuous functions.  However, 
Now, Kohlenbach has shown in \cite{kohlenbach4}*{\S4} that $\WKL$ suffices to prove that every continuous function on Cantor space has a code.  
Hence, assuming $\WKL$, a theorem in $\L_{\omega}$ about (higher-order) continuous functions (on Cantor space) does not really change if we introduce codes, i.e.\ there is a perfect match between the theorem expressed in $\L_{2}$ and $\L_{\omega}$.    
In other words, \emph{second-order} $\WKL$ (working in $\RCAo$) proves that the the second-order formalisation has the same scope as the original.  
In conclusion, $\L_{2}$ can talk about \emph{certain} continuous functions \emph{via codes}, and $\WKL$ (working in $\RCAo$) guarantees that the approach-via-codes actually is talking about \emph{all} (higher-order) continuous functions.  
In this light, Simpson's quote seems justified and correct.  

\smallskip

Our above results paint a different picture when it comes to measure theory: on one hand, a version of measure theory \emph{can} be expressed and developed in $\L_{2}$ using representations of measurable objects, as sketched in \cite{simpson2}*{X.1}.
On the other hand, if one wants the guarantee that the development in $\L_{2}$ using representations has the same \emph{scope or generality} as the original theory involving measurable objects, one needs to know that 
each measurable object has a representation.  To this end, one of course points to well-known approximation theorems like Lusin's.   However, the latter implies $\WHBU$ and is hard to prove in terms of conventional comprehension. 
In conclusion, $\L_{2}$ can talk about \emph{certain} measurable functions \emph{via codes}, but to know that the approach-via-codes actually is talking about \emph{all} measurable functions requires $\WHBU$ and hence $\Z_{2}^{\Omega}$, both of which are not part of second-order arithmetic.  

\smallskip

In conclusion, second-order arithmetic uses codes to talk about \emph{certain} objects of a given (higher-order) class, like continuous or measurable functions. 
However, to know that the $\L_{2}$-development based on codes has the same scope or generality as the original theory, one needs the guarantee that 
every (higher-order) object has a code.  Second-order arithmetic can apparently provide this guarantee in the case of continuous functions, \emph{but not in the case of measurable functions}.  Put another way, \emph{proving} that second-order arithmetic can fully express measure theory, seriously transcends second-order arithmetic.  

\begin{ack}\rm
We thank the anonymous referee for the many helpful suggestions that have markedly improved the paper. 
Our research was supported by the John Templeton Foundation, the Alexander von Humboldt Foundation, LMU Munich (via the Excellence Initiative and the Center for Advanced Studies of LMU), and the University of Oslo.
We express our gratitude towards these institutions. 
We thank Ulrich Kohlenbach, Karel Hrbacek, and Anil Nerode for their valuable advice.  
We thank Emil Je\v{r}\'{a}bek for the answer provided in \cite{jeranimo}.
%We also thank the anonymous referee for the helpful suggestions.  
Opinions expressed in this paper do not necessarily reflect those of the John Templeton Foundation.    
\end{ack}

\bye